%% file: positivecontrol.tex
\pdfoutput=1
\documentclass{shinyart}

\usepackage[utf8]{inputenc}

\usepackage{shinybib}

\usepackage{autonum}
\usepackage{dsfont}

\theoremstyle{definition}
\newtheorem{assumption}[theorem]{Assumption}

\newcommand{\N}{{\mathbb{N}}}
\newcommand{\R}{\mathbb{R}} 
\newcommand{\Rbar}{\overline\R{}} 

\renewcommand{\phi}{\varphi}
\newcommand{\eps}{\varepsilon}
\newcommand{\half}{\frac{1}{2}}
\newcommand{\calF}{\mathcal{F}}
\newcommand{\calG}{\mathcal{G}}

\newcommand{\omc}{{\overline\omega_c}}
\newcommand{\omo}{{\omega_o}}

\newcommand{\M}{\mathcal{M}}

\newcommand{\norm}[1]{\|#1\|}
\newcommand{\dual}[2]{\langle #1, #2\rangle}
\newcommand{\set}[1]{\left\{#1\right\}}

\DeclareMathOperator{\dom}{dom}
\DeclareMathOperator{\ran}{ran}

\newcommand{\pred}[1]{\prescript{*\!}{}{#1}}
\newcommand{\Aad}{\pred{\!A}}

\def\pdS{h} 
\def\piC{\phi} 

\usepackage{tikz,pgfplots}
\pgfplotsset{compat=newest}
\pgfplotsset{plot coordinates/math parser=false}

\title{Optimal control of elliptic equations with positive measures}

\author{Christian Clason%
    \thanks{Faculty of Mathematics, University Duisburg-Essen, 45117 Essen, Germany (\email{christian.clason@uni-due.de})}
    \and Anton Schiela%
    \thanks{Institute of Mathematics, University of Bayreuth, 95440 Bayreuth, Germany (\email{anton.schiela@uni-bayreuth.de})}
}
\date{September 7, 2015}

\hypersetup{
    pdftitle={Optimal control of elliptic equations with positive measures},
    pdfauthor={Christian Clason, Anton Schiela},
    pdfkeywords={Optimal control, measure control, control constraints, Fenchel duality, unbounded operators}
}

\begin{document}
\maketitle

%%%%%%%%%%%%%%%%%%%%%%%%%%%%%%%%%%%%%%%%%%%%%%%%%%%%%%%%%%%%%%%%%%%%%%%%%

\begin{abstract}
    Optimal control problems without control costs in general do not possess solutions due to the lack of coercivity. However, unilateral constraints together with the assumption of existence of strictly positive solutions of a pre-adjoint state equation, are sufficient to obtain existence of optimal solutions in the space of Radon measures.  Optimality conditions for these generalized minimizers can be obtained using Fenchel duality, which requires a non-standard perturbation approach if the control-to-observation mapping is not continuous (e.g., for Neumann boundary control in three dimensions).
    Combining a conforming discretization of the measure space with a semismooth Newton method allows the numerical solution of the optimal control problem.
\end{abstract}

\section{Introduction}

This work is concerned with the following optimal control problem, stated formally as
\begin{equation}\label{eq:problem}
    \inf_{y,u} \half \norm{Ey-y_d}_{L^2(\omo)}^2   
    \quad\text{s.\,t.}\quad  Ay-Bu=0,\quad u \ge 0,
\end{equation}
where $A$ is a second-order elliptic differential operator and $y_d$ is a given target.
Furthermore, $\omo\subset\overline\Omega\subset\R^d$ is the observation domain with corresponding restriction operator $E$, and the control is defined on a control domain $\omc\subset\overline\Omega$ with corresponding extension operator $B$. (This setting includes boundary control and observation; for details we refer to \cref{sec:state_equation}.)

Problem~\eqref{eq:problem} differs from standard control-constrained optimal control problems by the fact that no control cost term, e.g., of the form $\alpha \|u\|_{U}^2$ or $\|u\|_U$ with $\alpha > 0$ and a suitable Banach space $U$, appears in the functional.
This term is usually necessary to guarantee existence of an optimal solution $(\bar y,\bar u)$, since it provides us with coercivity of the objective functional in the appropriate topology. 
Consequently, one of the major issues in this work will be the discussion of existence of minimizers of this problem.
As we will show, the non-negativity together with the tracking term is sufficient (under an appropriate assumption on the operator $A$) to obtain coercivity with respect to $u$, albeit only in the space of measures. Intuitively, boundedness of $y=A^{-1}Bu$ in $L^2$ implies boundedness of $Bu$ only in $H^{-2}$, which is all one can expect in general without control constraints. It is thus surprising that in many cases optimal controls exist in the more regular space $\M$ of Radon measures if merely unilateral constraints are present, thus allowing to formulate, analyze and numerically solve the limit problem as $\alpha\to 0$ in the above-mentioned standard problems with unilateral constraints, which is the main motivation of this work.

Once existence of optimal controls is established, first-order optimality conditions can be derived via Fenchel duality. This is relatively straightforward in those cases where 
the control-to-observation mapping $u\to Ey$ is continuous as a mapping $\M(\omc)\to L^2(\omo)$. However, due to the low regularity of the control, this assumption is not satisfied for all relevant applications (e.g., Neumann-control in three dimensions; similar difficulties are to expected for parabolic problems). These cases require special care since they involve unbounded operators. A second motivation of this work is therefore to extend the Fenchel duality theorem to this setting.

Let us remark on some related problems. Recently, a class of elliptic problems came into the focus of interest, where control costs of the form $\alpha \|u\|_{L^1}$ were used and which possess generalized solutions $u\in \mathcal M$; see \cite{Clason:2010a,Clason:2011a,Clason:2012,Casas:2013}. In particular, we rely on the first three works for the numerical computation of our optimal measure space controls using a semismooth Newton method and a conforming finite element discretization of $\M$.
Often such functionals are still augmented by an additional $L^2$-type control cost as well as bilateral control constraints, and the limit $\beta \to 0$ is considered; see, e.g., \cite{Stadler:2007a,Wachsmuth:2009}. 
A second related problem class is that of so-called \emph{bang-bang}-problems \cite{Hinze:2012}, where no control costs are present, but the control constraints are bilateral, so that optimal solutions exist in $L^\infty$. 
Finally, due to the presence of measure-valued controls, we will have to define the operator $A$ in a way that $Ay=\mu$ has a unique solution for each $\mu\in \mathcal M$. This requires an extension of the usual variational setting in $H^1$. In this respect, our paper draws from results in the literature; see \cite{Schiela:2010} and the references therein. It also provides a link to the study of state-constrained problems \cite{Casas85}, where measure-valued right-hand sides appear in first-order optimality conditions.

This work is organized as follows. \Cref{sec:state_equation} discusses well-posedness of the state equation for measure-valued right-hand sides. In \cref{sec:existence:minimizer}, we give a  rigorous statement of Problem \eqref{eq:problem} and show that under a strict positivity assumption on the adjoint control-to-observation mapping, a minimizer to \eqref{eq:problem} exists in the space of Radon measures; we discuss the validity of this assumption in the context of second-order elliptic equations in ~\cref{sec:existence:slater}. \Cref{sec:existence:examples} gives some examples as well as a counterexample that shows the necessity of our assumption. Optimality conditions for these minimizers are derived in \cref{sec:optimality} based on a Fenchel duality theorem for an unbounded operator. In \cref{sec:MY}, we remark on the relation of Problem \eqref{eq:problem} to the corresponding problems including additional $L^2$ or measure-space control costs.
The numerical solution based on a variational discretization and a semismooth Newton method is discussed in \cref{sec:discretization}.
Finally, numerical examples are presented in \cref{sec:examples}.

\section{State equation}\label{sec:state_equation}

We first discuss well-posedness of the control-to-observation mapping $u\mapsto Ey$. Since $u$ is only a Radon measure and $E$ need not be continuous, 
this requires some technicalities. In particular, due to the presence of the non-reflexive spaces $C$ and $\M$ it will be useful to start with defining the \emph{pre-adjoint} operators of $A$ and $B$.

\paragraph{Elliptic differential operator $A$}
Consider a bounded domain (i.e., an open connected subset) $\Omega \subset \mathbb{R}^d$ with Lipschitz boundary $\partial\Omega$, so that the trace operator 
$H^1(\Omega) \to L^2(\partial \Omega)$ is well-defined.
Let $a(\cdot,\cdot) : H^1(\Omega) \times H^1(\Omega) \to \mathbb{R}$ be a continuous and elliptic bilinear form, defined by 
\begin{equation}\label{eq:general_bilinear_form}
    a(y,p):= \int_{\Omega} \left[\sum_{i,j=1}^d a_{ij}(x) y_{x_i}
    p_{x_j} +  c(x) y p\right]\, dx+\int_{\partial \Omega} r(x)y p\, ds.
\end{equation}
where subsequently we assume that the coefficients are symmetric (i.e.,  $a_{ij}=a_{ji}$) and bounded on $\Omega$, and that $c$ and $r$ are non-negative bounded functions in $\Omega$ and $\partial \Omega$, respectively.
Furthermore, assume that there exists $a_0>0$ such that
\begin{equation}
    \sum_{i,j=1}^d a_{ij}(x) \xi_i \xi_j \geq a_0 | \xi |^2 \quad \text{ for all } \xi \in \mathbb{R}^d
    \text{ and almost all } x \in \Omega.
\end{equation}
We assume further that not both $c$ and $r$ are identically $0$. As usual, it follows by the Poincar\'e inequality that $a$ is coercive, i.e., there exists $c_1 >0$ such that
\begin{equation}  \label{eq:coercive}
    a(y,y) \geq c_1 \norm{y}_{H^{1}(\Omega)}^2 \qquad \text{ for all } y \in H^1(\Omega).
\end{equation}
Alternatively, we could impose Dirichlet boundary conditions on (part of) $\partial \Omega$ to obtain coercivity. However, in the following
discussion we stick to the case $H^1(\Omega)$, mainly for simplicity of presentation. 

It then follows from the Lax--Milgram theorem that for each $\ell \in H^1(\Omega)^*$, there is a unique $y \in H^1(\Omega)$, such that
$a(y,p)=\ell(p)$ for all $p \in H^1(\Omega)$. In this way, the well-known isomorphism $A_{H^1} : H^1(\Omega) \to H^1(\Omega)^*$ is
constructed via $(A_{H^1}y)(p) := a(y,p)$. 

\paragraph{Extension to measure-valued right-hand sides}
Our next aim is to define a version of this operator that covers elliptic PDEs with measure-valued right-hand sides. For $d\ge 2$, this does not fit
into the classical variational framework. Following the method of Stampacchia \cite{Stampacchia:1965a}, we will therefore first construct an unbounded pre-dual operator $\pred A$ with 
domain $C(\overline \Omega)$, and then consider its adjoint $A:= (\pred A)^*$ whose co-domain is then -- by definition -- the dual of $C(\overline \Omega)$,
which can be identified by the Riesz representation theorem with the space of Radon measures $\M(\overline \Omega)$. The following construction is similar to the one
given in \cite{Schiela:2010}; our main reference concerning unbounded operators is \cite{Goldberg:2006}.

Consider an index $q>d$ (the spatial dimension), so that $W^{1,q}(\Omega)\hookrightarrow C(\overline \Omega)$, and its dual index $q'$ which satisfies $q^{-1}+q'^{-1}=1$. 
By H\"older's inequality applied to the derivatives,
$a(\cdot,\cdot)$ is still well-defined and continuous as a bilinear form 
\begin{equation}\label{eq:bilinear_W1q}
    a(\cdot,\cdot) : W^{1,q'}(\Omega) \times W^{1,q}(\Omega) \to \mathbb{R}.
\end{equation}

Let us define a domain $\dom \Aad\subset H^1(\Omega)$ (often called ``maximal domain of definition'') and a bijective mapping $\Aad: \dom\Aad \to W^{1,q'}(\Omega)^*$ in the following
way:
\begin{equation}\label{eq:dompredA}
    \dom\Aad:=\set{p\in H^1(\Omega): \exists\, c_p \in \mathbb{R} \text{ with } a(y,p) \leq c_p \norm{y}_{W^{1,q'}(\Omega)} \; \forall \, y\in H^1(\Omega)}.
\end{equation}
Let us stress that here (and in similar occasions) the bound $c_p$ may depend on $p$ but not on $y$. 

By \eqref{eq:bilinear_W1q}, we conclude that $H^1(\Omega) \supset \dom\Aad \supset W^{1,q}(\Omega)$, and under relatively mild assumptions
on the smoothness of the coefficients and on the domain, regularity theory even yields $\dom\Aad = W^{1,q}(\Omega)\hookrightarrow C(\overline \Omega)$ if $q$ is sufficiently close
to $d$; see, e.g., \cite[Theorem 3.16]{Troianiello:1987a}. This is called the case of ``maximal regularity''. In fact, for $d=2$, it is always possible to find an appropriate $q$. In this case we can define $\pred A$ as follows:
\begin{equation}
    \begin{aligned}
        \Aad&: C(\overline \Omega) \supset W^{1,q}(\Omega) \to W^{1,q'}(\Omega)^*,\\
        p &\mapsto \Aad p \; : (\Aad p)(y) := a(y,p) \quad \text{for all } y\in W^{1,q'}(\Omega).
    \end{aligned}
\end{equation}
Otherwise, if $\dom\Aad$ is a proper superset of $W^{1,q}(\Omega)$, the bilinear form $a(y,p)$ is not defined anymore for all $y\in W^{1,q'}(\Omega)$ and $p\in \dom\Aad$
due to lack of integrability of the principal part. However, by the definition of $\dom \Aad$ in \eqref{eq:dompredA}, we can extend $a(\cdot,\cdot)$ to
a bilinear form $\overline{a}(\cdot,\cdot): W^{1,q'}(\Omega) \times \dom\Aad$ via the unique continuous extension
\begin{equation}
    \overline{a}(y,p):= \lim_{n \to \infty}a(y_n,p) \quad \text{for all } (y,p) \in  W^{1,q'}(\Omega) \times \dom\Aad,
\end{equation}
where $\{y_n\}_{n\in\N}$ is a sequence in $H^1(\Omega)$ such that $y_n \to y$ in $W^{1,q'}(\Omega)$. By density of $H^1(\Omega)$ in $W^{1,q'}(\Omega)$, such
a sequence always exists, and by definition of $\dom\pred{A}$ in \eqref{eq:dompredA}, the limit of $a(y_n,p)$ always exists and depends only on the limit $y$.

Under very mild assumptions,
it is still possible to show $\dom\Aad \subset C(\overline \Omega)$ (see, e.g., \cite[Theorem~3.3, Corollary~3.5, Corollary~3.6]{Rehberg:2009}), 
so that we obtain:
\begin{equation}
    \begin{aligned}
        \Aad &: C(\overline \Omega) \supset \dom\Aad \to W^{1,q'}(\Omega)^*,\\
        p &\mapsto \Aad p \; : (\Aad p)(y) := \overline{a}(y,p) \quad \text{for all } y\in W^{1,q'}(\Omega).
    \end{aligned}
\end{equation}
In both cases $\Aad$ is a bijective, closed, unbounded operator (cf. \cite{Schiela:2010}) and thus has
continuous inverse $\Aad^{-1}$ by the open mapping theorem for closed operators; see, e.g., \cite[II.1.8]{Goldberg:2006}. 
In what follows only this -- more general -- setting is required, keeping in mind, however, that $\Aad$ (and thus 
also its adjoint, defined next) corresponds to $\overline{a}(\cdot,\cdot)$, which only coincides with $a(\cdot,\cdot)$ if
$\dom\Aad = W^{1,q}(\Omega)$, cf. \cite{Schiela:2010}. 

Since $\dom\Aad\supset W^{1,q}(\Omega)$ is dense in $C(\overline \Omega)$,  the Banach space adjoint (also called conjugate) $A := (\Aad)^*$ of $\Aad$ is well-defined 
as a linear operator (cf., e.g., \cite[Def. II.2.2]{Goldberg:2006})
\begin{equation}
    A :  W^{1,q'}(\Omega) \supset \dom A \to \M(\overline \Omega),
\end{equation}
where $\dom A$ is canonically defined as
\begin{equation}
    \dom A := \set{ y\in W^{1,q'}(\Omega) :  \exists\, c_y\in \mathbb{R} \text{ with } (\pred Ap)(y) = \overline{a}(y,p) \le c_y\|p\|_{C(\overline \Omega)} \quad \forall\, p\in \dom \Aad}.
\end{equation}
Then for any $y\in \dom A$, the mapping $p\mapsto\overline{a}(y,p)$ defines a continuous linear functional on the dense subspace $\dom\Aad\subset C(\overline \Omega)$. It can thus be extended uniquely to a continuous functional $Ay$ on $C(\overline \Omega)$ satisfying $(Ay)(p) = \overline{a}(y,p)$ for all $p\in \dom \Aad$. 
By the Riesz representation theorem, $Ay$ can be identified with an element of $\M(\overline \Omega)$. We stress that this is the standard construction of
the Banach space adjoint of an unbounded, densely defined operator. 
By \cite[Theorem II.2.6, Theorem II.4.4]{Goldberg:2006}, the operator $A$ is also closed and continuously invertible, because $\pred A$ is. 

We even obtain the following compactness property:
\begin{lemma}[\protect{\cite[Lemma 2.15]{Schiela:2010}}]\label{lem:wsd}
    Consider a sequence $\{\mu_n\}_{n\in\N}$ that converges weakly-$*$ in $\M(\overline\Omega)$ to $\mu$. Then the sequence $\{A^{-1}\mu_n\}_{n\in\N}$ 
    converges strongly in $W^{1,q'}(\Omega)$ to $A^{-1}\mu$.
\end{lemma}

\paragraph{Control operator $B$}
Next, consider a compact set $\omc\subset\overline\Omega$ such that there exists a 
continuous trace or embedding operator $\pred{B}_{H^1} : H^1(\Omega) \to L^2(\omc)$. Here $L^2(\omc)$ is defined with respect to an appropriate positive and bounded
measure $\nu$ on $\omc$; e.g., $\omc=\overline \Omega$ with the Lebesgue measure for distributed control, and $\omc=\partial \Omega$ with the boundary measure
for boundary control. Technically, we will require in the following that $\nu(\omc \cap O) > 0$ for any open subset $O\subset \R^d$ such that $\omc \cap O$ is non-empty. 
This guarantees applicability of \cref{thm:wsd} (see \cref{sec:appendix}).

We introduce the linear and continuous restriction operator
\begin{equation}
    \pred{B}:C(\overline\Omega)\to C(\omc),\qquad
    (\pred{B}v)(x) = v(x) \quad \forall x \in \omc,
\end{equation}
which coincides with the above mentioned restriction operator $\pred{B}_{H^1}$ on
$C(\overline \Omega)\cap H^1(\Omega)$, this space being dense in both $C(\overline \Omega)$ and $H^1(\Omega)$. 

Its adjoint $B := (\pred{B})^*$ can be interpreted (via the Riesz representation theorem) as a mapping
\begin{equation}
    B : \M(\omc) \to \M(\overline \Omega)
\end{equation}
acting as the extension by $0$ of a measure on $\omc$ to a measure on $\overline \Omega$. On $L^2(\omc)$ it coincides with 
the operator $B_{H^1}:=(\pred{B}_{H^1})^* : L^2(\omc) \to H^1(\Omega)^*$. Moreover, by \cref{thm:wsd} the space $L^2(\omc)$ is weakly-$*$ sequentially dense in $\M(\omc)$.

\paragraph{Observation operator $E$}
For the operator $E$, which will be defined on reflexive spaces, it is most convenient to start with the primal operator.
Let $\omo\subset\overline\Omega$, equipped with a suitable measure, and assume that there exists a closed (possibly unbounded) operator
\begin{equation}
    E : W^{1,q'}(\Omega) \supset \dom E \to L^2(\omo),
\end{equation}
where $\dom E \supset H^1(\Omega)$ is dense in $W^{1,q'}(\Omega)$. By this assumption, the restriction of $E$ to $H^1(\Omega)$, i.e.,
\begin{equation}
    E_{H^1} := E|_{H^1} : (H^1(\Omega),\|\cdot\|_{H^1}) \to L^2(\omo),
\end{equation}
is defined on all of $H^1(\Omega)$. It is readily verified that $E_{H^1}$ is closed as well. 
Thus, by the closed graph theorem (see, e.g., \cite[II.1.9]{Goldberg:2006}), $E_{H^1}$ is even a continuous operator. 

In many cases $E$ is continuous for suitable $q'$, and $\dom E = W^{1,q'}(\Omega)$ holds, 
but there are also important cases where $E$ lacks continuity. Typical examples (e.g., embedding or trace operators) are discussed
in detail below. 

By reflexivity, we can define its adjoint $\pred{E}:=E^*$ as a closed operator
\begin{equation}
    \pred{E} : L^2(\omo) \supset \dom \pred{E} \to W^{1,q'}(\Omega)^*,
\end{equation}
since in this case $(\pred{E})^*=E^{**}=E$.
Like all adjoints of closed operators in reflexive spaces, $\pred{E}$ has a dense domain; see, e.g., \cite[Theorem II.2.14]{Goldberg:2006}. 
Comparison with $\pred{E_{H^1}} := E^*_{H^1}$ yields
that $\pred{E_{H^1}}h=\pred{E}h$ for every $h$ for which the latter is defined, i.e., for $h\in \dom\pred E$. Thus, the continuous operator $\pred{E}_{H^1}$ can be considered as the unique continuous extension
of $\pred{E}$ after the co-domain space has been extended from $W^{1,q'}(\Omega)^*$ to $H^1(\Omega)^*$ (and renormed).

\paragraph{Control-to-observation mapping $S$}
Finally, we define 
\begin{equation}
    \pred{S} : L^2(\omo) \supset \dom \pred{S} \to C(\omc),\qquad
    \pdS \mapsto \pred{B}\Aad^{-1}\pred{E} \pdS,
\end{equation}
where $\dom\pred S := \dom \pred{E}$ is dense in $L^2(\omo)$ by our above assumptions. This mapping is well-defined, since $\pred{B}\Aad^{-1} : W^{1,q'}(\Omega)^*\to C(\omc)$ is a continuous operator, defined on all of $W^{1,q'}(\Omega)^*$.
Since the adjoint of a densely defined (unbounded) linear operator is closed, see, e.g., \cite[Theorem II.2.6]{Goldberg:2006}, $S:=(\pred{S})^*$ is a closed operator
\begin{equation}
    S :\M(\omc)\supset \dom S \to L^2(\omo).
\end{equation}

Since $E$ may be unbounded, the following assertion is not obvious.
\begin{lemma}\label{lem:propS}
    It holds that
    \begin{equation}\label{eq:domS}
        \dom {EA^{-1}B} := \set{ u \in \M(\omc) : A^{-1}B u \in \dom {E}} = \dom {S} \supset L^2(\omc).
    \end{equation}
    and $S = EA^{-1}B$. 
    Furthermore, $S$ is weakly-$*$ closed, i.e., if $u_n \rightharpoonup^* u$ in $\M(\omc)$ and $h_n \rightharpoonup h$ in $L^2(\omo)$ with $Su_n = h_n$, then $Su = h$. 
\end{lemma}
\begin{proof}
    By purely algebraic arguments we have for $u\in \dom S \cap \dom EA^{-1}B$ that $Su =EA^{-1}Bu$ since then both sides
    of the equality are well-defined. Thus, we have to prove the equality of their domains, using the definition of $\dom EA^{-1}B$
    in \eqref{eq:domS}. 
    By continuity of $\pred{B}\Aad^{-1}$ we conclude 
    \begin{equation}\label{eq:Sphimu}
        \dual{u}{\pred{S}\pdS}_{\M(\omc),C(\omc)}  = \dual{A^{-1}Bu}{\pred{E}\pdS}_{W^{1,q'}(\Omega),W^{1,q'}(\Omega)^*} \quad \text{ for all } \pdS \in \dom \pred{S}, \  u \in \M(\omc).
    \end{equation}
    By definition of domains of adjoints, $u \in \dom S$ iff $\dual{u}{\pred{S}\pdS}_{\M(\omc),C(\omc)} \le c_u \|\pdS\|_{L^2(\omo)}$, and 
    $A^{-1}B u \in \dom {E}$ iff $\dual{A^{-1}Bu}{\pred{E}\pdS}_{W^{1,q'}(\Omega),W^{1,q'}(\Omega)^*} \le c_{A^{-1}Bu}\|\pdS\|_{L^2(\omo)}$. By \eqref{eq:Sphimu},
    $c_u= c_{A^{-1}Bu}$, and hence the domains coincide. 

    The last inclusion in \eqref{eq:domS} follows from 
    the fact that for  $u \in L^2(\omc)$, we have $A^{-1}Bu \in H^1(\Omega)\subset \dom E$. 
    This in turn is a consequence of $Bu \in H^1(\Omega)^*$, so that $A^{-1}Bu$ coincides with the variational solution of the state equation. 

    By \cref{lem:wsd}, weak-$*$ convergence of $u_n$ implies strong convergence of $A^{-1}B u_n$ in $W^{1,q'}(\Omega)$. Since $E$ is closed, it is
    also weakly closed (since its graph is a convex closed set, thus weakly closed). 
    Hence, $A^{-1}B u_n \to A^{-1}Bu$ and $h_n \rightharpoonup h$ with $Su_n = h_n$ imply $Su = EA^{-1}Bu = h$. 
\end{proof}

We remark for later reference that by definition of adjoints, we have that
\begin{equation}
    \dual{Su}{\pdS}_{L^2(\omo)} = \dual{u}{\pred{S}\pdS}_{\M,C}\quad\text{ for all } u\in\dom S,\ \pdS\in\dom \pred{S},
\end{equation}
where here and in the following, we have omitted the domains from the spaces appearing in duality pairings if they are clear from the context.
Also, by definition of $\dom S$, for $u\notin\dom S$ there exists a bounded sequence $\pdS_n$ in $\dom\pred{S}$ such that $\dual{u}{\pred{S}\pdS_n}_{\M,C}\to \infty$. 

Finally, we remark that $\dom S$ is weak-$*$ sequentially dense in $\M(\omc)$. This follows via $\dom S \supset L^2(\omc)$, using 
\cref{thm:wsd}, which states that $L^2(\omc)$ is weakly-$*$ sequentially dense in $\M(\omc)$. 
In particular, $\langle u, \piC\rangle_{\M,C}=0$ for all $u \in \dom S$ implies $\langle u,\piC\rangle_{\M,C}=0$ for all $u \in \M(\omc)$ and thus
$\piC=0$ as an element of $C(\omc)$. 

Using $B_{H^1}$ and $E_{H^1}$, we complement the measure-space operators $S$ and $\pred{S}$ by their ``standard'' counterparts, i.e., the continuous mappings
\begin{equation}
    S_{H^1} := E_{H^1}A_{H^1}^{-1}B_{H^1} : L^2(\omc) \to L^2(\omo) \quad \mbox{ and } \quad \pred{S}_{H^1}:= S^*_{H^1} : L^2(\omo) \to L^2(\omc).
\end{equation}
The operator $S_{H^1}$ is a
restriction of $S$ and coincides with it on $L^2(\omc)$. In contrast, $\pred{S}_{H^1}$ is an extension of $\pred{S}$ and
is defined on all of $L^2(\omo)$ and not only on $\dom \pred{S}$.
This is possible because $\pred{S}_{H^1}$ has a larger co-domain $L^2(\omc) \supset C(\omc)$. 

\section{Existence of minimizers}\label{sec:existence:minimizer}

Using the control-to-observation operator, we can state Problem \eqref{eq:problem} in reduced form as 
\begin{equation}\label{eq:prob_reduced}
    \min_{u\in\M(\omc)} \half \norm{Su-y_d}_{L^2(\omo)} + \delta_{\M(\omc)^+}(u),
    \tag{P}
\end{equation}
where $\delta_{\M(\omc)^+}$ denotes the indicator function of the positive cone in $\M(\omc)$, i.e., 
\begin{equation}
    \M(\omc)^+ := \set{u \in \M(\omc):\dual{u}{\piC}_{\M,C}\geq 0\text{ for all }\piC\in C(\omc), \piC\geq 0}.
\end{equation}
We now address existence of minimizers to \eqref{eq:prob_reduced}, which requires an assumption on the control-to-observation operator which we call a \emph{pre-dual Slater condition}. 
Since this operator is defined via duality, it will be seen that it is natural to formulate this assumption in terms of the pre-adjoint~$\pred{S}$.
\begin{assumption}[Pre-dual Slater condition]\label{ass:key}
    There exists a function $\pdS\in\dom \pred{S}\subset L^2(\omo)$ such that $\pred{S}\pdS\in C(\omc)$ is strictly positive, i.e., there is $\eps >0$ such that
    \begin{equation}\label{eq:predslater}
        (\pred{S}\pdS)(x) \geq \eps >0 \qquad\text{ for all } x\in\omc.
    \end{equation}
\end{assumption}
Since $\pred{S}=\pred{B}\Aad^{-1}\pred{E}$, \cref{ass:key} claims the existence of a function $\pdS\in L^2(\omo)$ such that the solution $p$ of the equation $\Aad p=\pred{E}\pdS$ is a continuous function and satisfies $\pred{B}p \ge \eps>0$. 
We are thus looking for solutions of elliptic equations that are \emph{strictly} positive (on parts of the domain). 

Using this assumption, we can show that a minimizing sequence is bounded in a sufficiently strong topology.
\begin{lemma}\label{lem:bounded} If \cref{ass:key} holds, then
    any minimizing sequence  $\{u_n\}_{n\in\N} \subset \M(\omc)$  for \eqref{eq:prob_reduced} is bounded in $\M(\omc)$ with $\{S u_n\}_{n\in\N}$ bounded in $L^2(\omo)$.
\end{lemma}
\begin{proof}
    First, note that the non-negativity constraint and coercivity of the tracking term imply, respectively, that $u_n\geq 0$ for all $n\in\N$ and that $\{Su_n\}_{n\in\N}$ is bounded in $L^2(\omo)$ (and in particular, that $\{u_n\}_{n\in\N}\subset\dom S$).
    Using \cref{ass:key} and identifying $\eps>0$ with the constant function $\eps\mathds{1}(x)\in C(\omc)$, we thus deduce from the definition of the total variation norm of a non-negative measure that
    \begin{equation}
        \begin{aligned}
            \eps\norm{u_n}_{\M(\omc)} &= \eps\int_{\omc}\,du_n = \dual{u_n}{\eps}_{\M,C} \leq  \dual{u_n}{\pred{S}\pdS}_{\M,C} =   \dual{Su_n}{\pdS}_{L^2(\omo)} \\
                                      &\leq \norm{Su_n}_{L^2(\omo)} \norm{\pdS}_{L^2(\omo)}\leq C,
        \end{aligned}
    \end{equation}
    and hence the claimed boundedness follows.
\end{proof}
With this, we obtain existence of a minimizer by Tonelli's direct method.
\begin{theorem}\label{thm:existence}
    Under the above assumptions, there exists a minimizer $\bar u\in\M(\omc)$ of \eqref{eq:prob_reduced} such that $S\bar u \in L^2(\omo)$. 
    If $S$ is injective, $\bar u$ is unique.
\end{theorem}
\begin{proof}
    Let  $\{u_n\}_{n\in\N} \subset \M(\omc)$ be a minimizing sequence for \eqref{eq:prob_reduced}, which is bounded in $\M(\omc)$ by \cref{lem:bounded}. Since $C(\omc)$ is separable, the Banach--Alaoglu theorem yields 
    existence of a subsequence converging weakly-$*$ to some $\bar u\in\M(\omc)$. By boundedness of $Su_n$, we may then extract another subsequence such that $Su_n$ converges weakly to some $z \in L^2(\omo)$. By \cref{lem:propS} we obtain $z=S\bar u$. From weak-$*$
    sequential closedness of the non-negative cone in $\M$, we deduce that $\bar u$ is feasible and thus a minimizer of \eqref{eq:prob_reduced}. Finally, strict convexity of the tracking term implies that any pair of minimizers $u_1,u_2$ satisfies $S u_1 = S u_2$ and hence, if $S$ is injective, $u_1 = u_2$.
\end{proof}

\subsection{Verification of the pre-dual Slater condition}\label{sec:existence:slater}

We now discuss situations in which \cref{ass:key} can be verified. Recall that we have to show for some $\pdS\in \dom \pred{E}$ the existence of a solution $p\in\dom\Aad$ to the equation
\begin{equation}
    a(y,p)=\dual{\pdS}{E y}_{L^2(\omo)} \quad \text{ for all } y\in \dom E 
\end{equation}
such that $\pred{B}p$ is strictly positive on $\omc$. Although it is well-known that elliptic PDEs have non-negative solutions for non-negative right-hand sides
and boundary data, existence of a strictly positive solution is not a trivial matter and of course not satisfied in general (consider the homogenous Dirichlet
problem and $\omc=\overline \Omega$). Moreover, the literature -- although quite exhaustive for the Dirichlet problem -- is much scarcer in the case of Neumann, Robin or even mixed boundary conditions. 

We first remark that under the stated assumptions, $a(\cdot,\cdot)$ given by \eqref{eq:general_bilinear_form} is uniformly elliptic and hence defines a positive operator, i.e., for all $p\in H^1(\Omega)$,
\begin{equation}
    a(y,p) \ge 0 \quad \text{ for all }\quad y \in H^1(\Omega), y\geq 0 \quad \Rightarrow\quad p \ge 0.
\end{equation}
This already implies strict positivity on compact subsets of $\Omega$.
\begin{lemma}\label{lem:Harnack}
    Let $\Omega \subset \mathbb{R}^d$ be a domain. 
    Assume that $p\ge 0 \in H^1(\Omega) \cap C(\overline \Omega)$ satisfies $p\not\equiv 0$ and
    \begin{equation}\label{eq:vpositive}
        a(y,p) \ge 0 \quad \text{ for all }\quad  y\in H^1_0(\Omega), y\geq 0.
    \end{equation}
    If $K \subset \Omega$ is compact, there is a $\delta > 0$ such that $p \ge \delta$ on $K$, and in particular, $p>0$ on~$\Omega$.
\end{lemma}
Note the discrepancy between $p\in H^1(\Omega)$ and $y\in H^1_0(\Omega)$; we choose this setting because it fits
to the setting in \cite[Chapter 8]{GilTru1977}, from which we cite a crucial result: the Harnack inequality. Unfortunately, a Harnack inequality for the
setting $y\in H^1(\Omega)$ (covering Robin, Neumann, or mixed boundary conditions explicitly)
is hard to find in the literature.
\begin{proof}
    The result is a consequence of the weak Harnack 
    inequality (cf. \cite[Theorem 8.18]{GilTru1977}), which holds for non-negative supersolutions of $a(p,\cdot)=0$.  Let $x\in \Omega$ be given and denote by $B_{r}(x)$ a ball around $x$
    of radius $r$. If $B_{4R}(x)\subset \Omega$, then there exists a $C>0$ such that
    \begin{equation}\label{eq:weakharnack}
        C \inf_{B_{R}(x)}p \ge R^{-d}\|p\|_{L^1(B_{2R}(x))}. 
    \end{equation}

    With this result, we will show that either $p\equiv 0$ or $p>0$ on $\Omega$ for any supersolution $p\ge 0$. Since $\Omega$ is a domain, and thus open and connected, 
    we merely have to assert that $\Omega_0:=\{x\in \Omega : p(x)=0\}$
    is open and closed, because then either $\Omega_0=\Omega$ (i.e., $p\equiv 0$) or $\Omega_0=\emptyset$ (i.e. $p>0$).
    Indeed, by continuity of $p$, $\Omega_0$ is (relatively) closed in $\Omega$ and by \eqref{eq:weakharnack}, every $x\in \Omega_0$ is contained in a ball $B_{2R}(x)\subset \Omega_0$
    as long as $B_{4R}(x) \subset \Omega$. 
    Hence, $\Omega_0$ is open. 
    Thus, if $p\not\equiv 0$ on $\Omega$, we have $\Omega_0=\emptyset$ and so $p>0$ on $\Omega$. 

    Finally, if $K \subset \Omega$ is compact, then $p>0$ has a minimizer $\underline x$ on $K$, i.e., $p(x)\geq \delta:=p(\underline x)>0$ for all $x\in K$.
\end{proof}
In what follows we denote $L^s(\overline \Omega):=L^s(\Omega) \times L^s(\partial \Omega)$,
where the first factor is equipped with the Lebesgue measure, and the second with the boundary measure; we denote the corresponding 
product measure by $d\overline \nu := dx\times ds$. 
If $M$ is any subset of $\overline \Omega$, the space $L^s(M)$ is taken relatively to $L^s(\overline \Omega)$. 

\Cref{lem:Harnack} already yields a first result. In the following, $\chi_M$ denotes the characteristic function of $M$, 
which is identically $1$ on $M\subset \overline \Omega$ and $0$ on $\overline \Omega \setminus M$. 
\begin{corollary}
    If{} $\,\omc$ is a compact subset of $\Omega$ and $\omo \subset \overline \Omega$ has positive measure (i.e., $\overline \nu(\omo)>0$), then \cref{ass:key} is satisfied.
\end{corollary}
\begin{proof}
    Set $\pdS := \chi_{\omo}>0$ in \eqref{eq:predslater}. Since $\pdS\in L^\infty(\Omega \times \partial \Omega)\subset W^{1,q'}(\Omega)^*$, we have $\Aad^{-1}\pdS \in C(\overline \Omega)$ and thus
    $\pdS \in \dom \pred{S}\subset L^2(\omo)$. Hence, \cref{lem:Harnack} can be applied and yields the desired result. 
\end{proof}

Next, we want to cover the general case $\omc \subseteq\overline \Omega$.
\begin{lemma}\label{lem:strictlypos}
    Assume that $p\in H^1(\Omega)$ satisfies $p\not\equiv 0$ as well as 
    \begin{equation}\label{eq:strictlypos:obs}
        a(y,p)=\int_{\overline \Omega}\chi_{\omo} y\,d\overline \nu\quad \text{for all }y\in H^1(\Omega),
    \end{equation}
    and assume moreover that there is $\delta > 0$ such that for $(c,r)\in L^\infty(\Omega)\times L^\infty(\partial \Omega)$ it holds that
    \begin{equation}\label{eq:strictlypos:rob}
        \left\{\begin{aligned}
                c &= 0 \text{ on } (\Omega \setminus \omo) \cap \{ x \in \Omega : p(x) < \delta \}\\
                r &= 0 \text{ on } (\partial \Omega \setminus \omo) \cap \{ x \in \partial \Omega : p(x) < \delta \}.
        \end{aligned}\right.
    \end{equation}
    Then $p \ge \varepsilon := \min\left\{ \delta, \|r\|^{-1}_{L^\infty(\Omega)}, \|c\|^{-1}_{L^\infty(\Omega)}\right\}$. 
\end{lemma}
\begin{proof}
    We insert $y:=p^- := \min\{p,\varepsilon\}-\varepsilon \le 0$, which is in $H^1(\Omega)$, into \eqref{eq:general_bilinear_form} and show that $p^-=0$ and thus $p \ge \varepsilon$. 
    Observe that $p \le \varepsilon$ implies $p=p^-+\varepsilon$ and that $p> \varepsilon$ implies $p^- =0$ and $p^-_{x_i}=0$ for $i=1\dots d$. With this we compute:
    \begin{equation}
        \begin{aligned}
            \int_{\overline \Omega}\chi_{\omo} p^-\,d\overline \nu=a(p^-,p)&=\int_{\Omega} \sum_{i,j=1}^d a_{ij} p^-_{x_i} v_{x_j}
            +  c p^- p\, dx+\int_{\partial \Omega} r p^- p\, ds\\
            &=\int_{\Omega} \sum_{i,j=1}^d a_{ij} p^-_{x_i} p^-_{x_j}
            +  c p^- (p^-+\varepsilon)\, dx+\int_{\partial \Omega} r p^- (p^-+\varepsilon)\, ds\\
            &=a(p^-,p^-)+\varepsilon\left(\int_{\Omega}  c p^- \, dx+\int_{\partial \Omega} r p^- \, ds\right)
        \end{aligned}
    \end{equation}
    and obtain
    \begin{equation}
        \begin{aligned}
            0 &\le a(p^-,p^-)=\int_{\overline\Omega}\chi_{\omo} p^-\,d\overline \nu- \varepsilon \left(\int_\Omega c p^-\, dx+\int_{\partial \Omega} r p^-\, d s\right)\\
              &=\int_{\omo \cap \Omega}(1 - \varepsilon\, c) p^-\, dx+\int_{\omo \cap \partial \Omega}(1 - \varepsilon\, r) p^-\, ds
            -\varepsilon \left(\int_{\Omega\setminus \omo} c p^-\, dx+\int_{\partial \Omega\setminus \omo} r p^-\, ds\right).
        \end{aligned}
    \end{equation}
    Since $p \ge \delta \ge \varepsilon$ implies that $p^-=0$, the last two integrals vanish by our assumption on $c$ and $r$. Moreover, since 
    $1 - \varepsilon\, c \ge 1- \varepsilon\, \|c\|_{L^\infty(\Omega)} \ge 0$ and $1 - \varepsilon\, r \ge 1- \varepsilon\, \|r\|_{L^\infty(\partial \Omega)} \ge 0$, 
    the first two integrals are non-positive (recall that $p^-\le 0$).
    It follows that $a(p^-,p^-)=0$, implying $p^-=0$. 
\end{proof}
From this we can deduce the following sufficient criterion for the pre-dual Slater condition.
\begin{proposition}\label{pro:keyfulfilled}
    If $r=0$ on $\partial \Omega \setminus \omo$, then \cref{ass:key} is fulfilled for any compact $\omc\subset\overline\Omega$.
\end{proposition}
\begin{proof}
    We show that the solution $p$ of \eqref{eq:strictlypos:obs} is strictly positive. By \cref{lem:Harnack}, we already know that $p > 0$ on $\Omega$. 
    For $\delta > 0$, let $\Omega_\delta := \{ x\in \Omega : p(x) \le \delta \}$. Note that $|\Omega_\delta| \to 0$ as $\delta \to 0$ since $p>0$ on $\Omega$. 

    Define $a_\delta(\cdot,\cdot)$ like $a(\cdot,\cdot)$ but with $c$ replaced by $c_\delta := (1-\chi_{\Omega_\delta}) c$, and $p_\delta$ as the solution of 
    \begin{equation}
        a_\delta(y,p_\delta)=\int_{\overline \Omega} \chi_{\omo} y\, d\overline\nu\; \quad\text{for all } y \in H^1(\Omega).
    \end{equation}
    Then $p_\delta \ge 0$ and
    \begin{equation}
        \begin{aligned}
            a(y,p_\delta)&=a_\delta(y,p_\delta)+\int_{\Omega} \chi_{\Omega_\delta}cp_\delta \,dx=\int_{\overline \Omega} \chi_{\omo} y\, d\overline \nu+\int_{\Omega} \chi_{\Omega_\delta}cp_\delta \,dx\\
                         &=a(y,p)+\int_{\Omega} \chi_{\Omega_\delta}cp_\delta \,dx\\
                         &\ge a(y,p).
        \end{aligned}
    \end{equation}
    Hence, $p_\delta \ge p$, and thus $p_\delta(x) < \delta$ implies that $p(x) < \delta$ and thus $c_\delta(x) = 0$. Hence, \cref{lem:strictlypos} yields
    (after choosing $\delta \le \min\left\{\|c\|^{-1}_{L^\infty(\Omega)},\|r\|^{-1}_{L^\infty(\Omega)}\right\}$) that $p_\delta \ge \delta$. 

    Furthermore,
    \begin{equation}
        a(y,p-p_\delta)=\int_{\Omega} \chi_{\Omega_\delta} cp_\delta y \,dx,
    \end{equation}
    and for any $1\leq s<\infty$,
    \begin{equation}
        \|\chi_{\Omega_\delta} cp_\delta\|_{L^s(\Omega)} \le |{\Omega_\delta}|^{1/s}\|c\|_{L^\infty(\Omega)}\|p_\delta\|_{L^\infty(\Omega_\delta)},
    \end{equation}
    so that by \cite[Théorème 4.1]{Stampacchia:1965a}, there exists a $C>0$ such that for any $s>d$,
    \begin{equation}
        C \|p-p_\delta\|_{L^\infty(\Omega)} \le \|\chi_{\Omega_\delta} cp_\delta\|_{L^s(\Omega)} \le  |{\Omega_\delta}|^{1/s}\|c\|_{L^\infty(\Omega)}\|p_\delta\|_{L^\infty(\Omega_\delta)}. 
    \end{equation}
    Since $|\Omega_\delta|\to 0$ for $\delta \to 0$, 
    we can choose $\delta$ sufficiently small such that for adequately chosen $s\in(d,\infty)$, we have
    \begin{equation}
        C^{-1}\|c\|_{L^\infty(\Omega)} |{\Omega_\delta}|^{1/s} \le \frac14.
    \end{equation}
    Hence, we can estimate
    \begin{equation}
        \|p_\delta\|_{L^\infty(\Omega_\delta)} \le \|p\|_{L^\infty(\Omega_\delta)}+\|p-p_\delta\|_{L^\infty(\Omega)} \le \delta + \frac14\|p_\delta\|_{L^\infty(\Omega_\delta)},
    \end{equation}
    i.e., $\|p_\delta\|_{L^\infty(\Omega_\delta)} \le \frac{4}{3}\delta$. We conclude that $\|p-p_\delta\|_{L^\infty(\Omega)} \le \frac14 \frac43 \delta = \frac13\delta$, and therefore
    \begin{equation}
        p \ge p_\delta -\|p-p_\delta\|_{L^\infty(\Omega)} \ge \delta - \frac13 \delta > 0 
    \end{equation}
    as claimed.
\end{proof}

\subsection{Examples}\label{sec:existence:examples}

To illuminate our abstract framework further, let us discuss in the following a couple of examples. All of them have in common the generic definition of 
\begin{equation}
    A: W^{1,q'}(\Omega)\supset \dom A \to \M(\overline \Omega),
\end{equation}
where $q' \le 2$ is chosen appropriately as stated in the beginning of \cref{sec:state_equation}. 
However, the examples will cover different definitions of $E$ and $B$ and the corresponding spaces, i.e., different types of control and observation. 

\paragraph{Distributed control for a Neumann problem}\label{ex:dist_neumann}

As a first example, consider a homogeneous Neumann problem with distributed control (i.e., $r=0$ and $\omc = \overline \Omega$), such that 
\begin{equation}
    B = \mathrm{Id}: \M(\overline \Omega) \to \M(\overline \Omega)
\end{equation}
is the control operator with pre-adjoint $\pred{B} = \mathrm{Id}:C(\overline \Omega)\to C(\overline\Omega)$.  

Let us first consider boundary observation, i.e., $\omo = \partial \Omega$. We start with recalling that there exists a continuous trace operator
\begin{equation}
    \tau_{q'} : W^{1,q'}(\Omega) \to L^s(\partial \Omega)
\end{equation}
for suitably chosen $s$ depending on $q'$ and the spatial dimension $d$ of $\Omega$. In particular, for $q'=2$ we may always choose $s=2$. 
In the general case, we may define
\begin{equation}
    \dom E := \set{ y \in W^{1,q'}(\Omega) : \tau_{q'} y \in L^2(\partial \Omega)}
\end{equation}
(which implies $\dom E \supset H^1(\Omega)$ if $q' \le 2$), and then
\begin{equation}
    E : W^{1,q'}(\Omega) \supset \dom E \to L^2(\partial \Omega)
\end{equation}
as the restriction of $\tau_{q'}$ to $\dom E$. Since the norm of the co-domain space has been strengthened, $E$ is in general not continuous
anymore. It is, however, a closed operator: Assume that $y_n \to y$ in $W^{1,q'}(\Omega)$ and $E y_n \to h$ in $L^2(\partial \Omega)$. By
continuity of $\tau_{q'}$, we conclude that $Ey_n \to \tau_{q'} y$ in $L^s(\partial \Omega)$; but from $Ey_n \to h$ in $L^2(\partial \Omega)$
we deduce that $\tau_{q'} y=h\in L^2(\partial \Omega)$ and thus $y\in \dom E$ and $E y=\tau_{q'} y=h$. 

We summarize that $E$ satisfies all our assumptions, and note that for $d=2$ we may choose $q'$ sufficiently close to $2$ such that
$E:=\tau_{q'} : W^{1,q'}(\Omega) \to L^2(\partial \Omega)$ is well-defined as a continuous operator. However, the same is impossible
for $d=3$, so that we have to work with unbounded $E$ in this case. 

For the case of observation on the whole domain (i.e., $\omo=\Omega$) and $d\le 3$, we may simply define $E: W^{1,q'}(\Omega) \to L^2(\Omega)$ as the Sobolev embedding
which exists for suitably chosen $q'$. In the ``exotic'' case $d>3$, a similar effect as for boundary control with $d=3$ appears, and $E$ has to be defined
as an unbounded operator.

By \cref{pro:keyfulfilled} and by our assumption $r=0$, we see that we can choose $\omo \subset \overline \Omega$ arbitrarily as long as it has positive measure
with respect to the measure $d\overline \nu$ on $\overline \Omega$.

\paragraph{Robin or Neumann boundary control}\label{ex:bound}

In this case, our control operator is defined as the extension by zero 
\begin{equation}
    B : \M(\partial \Omega) \to \M(\overline \Omega),
\end{equation}
i.e., $\pred{B}:C(\overline\Omega)\to C(\partial\Omega)$ denotes the trace operator from $\overline\Omega$ to $\omc=\partial\Omega$. Again, we take $\pred{E}$ as the identity. 
To verify the pre-dual Slater condition, we then need to find $\pdS\in L^2(\Omega)$, such that the solution $p\in W^{1,q}(\Omega)$ of the problem
\begin{equation}
    a(y,p)= \dual{\pdS}{E y}_{L^2(\omo)} \quad\text{for all } y\in W^{1,q'}(\Omega)
\end{equation}
has a strictly positive boundary trace, i.e., $\pred{B}p \ge \varepsilon > 0$. According to \cref{pro:keyfulfilled} this can be achieved for Neumann boundary conditions
if $\omo$ is arbitrary (of non-zero measure), and for Robin boundary conditions if $\omo \supset \partial \Omega$. 

\paragraph{Distributed control for a Dirichlet problem}

We close this section with a simple example for which \cref{ass:key} is violated. 
Consider the problem
\begin{equation}\label{eq:counter}
    \left\{ 
        \begin{aligned}
            &\min J(y) := \|y-(1-x)\|^2_{L^2([0,1])} \quad\text{s.\,t.}\quad  u \ge 0,\\
            & -y''=u,\quad y(0) = y(1)=0.
        \end{aligned}
    \right.
\end{equation}
Due to the homogemous Dirichlet boundary conditions and by continuity, there cannot be any solutions of the predual problem which are larger than some $\varepsilon > 0$
on the whole domain, which coincides with the control domain. So \cref{ass:key} is clearly violated.

To show that also the conclusions of Theorem 3.3 do not hold, let us take for $n\ge 2$ the sequence of measures $u_n=n \delta_{1/n}$, 
which is contained in $\M([0,1])$ but unbounded. 
\begin{lemma}\label{lem:counter}
    The weak solution $y_n\in H^1_0(0,1)$ of $y'' =n \delta_{1/n}$ is given by
    \begin{equation}
        y_n =
        \begin{cases}
            (n-1) x & x\le 1/n,\\
            1-x &  x \ge 1/n.
        \end{cases}
    \end{equation}
\end{lemma}
\begin{proof}
    We have to find $y_n$ such that $\int_\Omega y_n' p' \,dx=n\,p(1/n)$ for all $p\in H^1_0((0,1))$ and $y_n(0)=y_n(1)=0$. 
    By the Lax--Milgram theorem, we know that this solution is unique; moreover, the special form of the right-hand side leads us to the ansatz $y'_n=\alpha$ on $[0,1/n]$ and $y'_n=\beta$ on $[1/n,1]$. 
    Using the homogenous boundary conditions, we find that $y_n=\alpha x$ on $[0,1/n]$ and $y_n=\beta(x-1)$ on $[1/n,1]$. 
    Since $y_n$ has to be continuous at $x=1/n$, we conclude that $\alpha\frac1n=\beta\frac1{n-1}$.

    Then, we can obtain using the weak formulation and the fundamental theorem of calculus that
    \begin{equation}
        \begin{aligned}
            \langle u_n,p\rangle_{\M,C} = n p(1/n)&=\int_0^{1/n}\alpha p'\,dx+\int_{1/n}^1 \beta p'\,dx\\
                                                  &= \alpha(p(1/n)-p(0))+\beta(p(1)-p(1/n))\\
                                                  &=(\alpha-\beta)p(1/n),  
        \end{aligned}
    \end{equation}
    which implies that $\alpha-\beta = n$. 
    Solving these two equations for $\alpha$ and $\beta$ yields our claim. 
\end{proof}

\begin{proposition}
    Problem~\eqref{eq:counter} does not possess an optimal solution in $\mathcal M([0,1])$. 
\end{proposition}
\begin{proof}
    From \cref{lem:counter} we conclude that $y_n \to 1-x$ in $L^2((0,1))$. 
    Hence, $\{(y_n,u_n)\}_{n\in\N}$ is a minimizing sequence, since each pair is feasible and $J(y_n)\to 0\leq J(y)$ for all $y$. 
    However, the limit $J=0$ cannot be attained, because the only possible candidate $y(x)=1-x$ does not satisfy the boundary conditions. 
\end{proof}

If we instead consider
\begin{equation}
    \left\{ \begin{aligned}
            &\min \|y-(1-x)\|^2_{L^2([\delta,1-\delta])} \quad\text{s.\,t.}\quad  u \ge 0,\\
            & -y''=u,\quad y(0) = y(1)=0,
    \end{aligned}\right.
\end{equation}
for some $\delta> 0$, then the control domain $[\delta,1-\delta]$ is a compact subset of $(0,1)$. So by \cref{lem:Harnack} we can verify \cref{ass:key} and
thus apply \cref{thm:existence} to assert existence of an optimal control in $\mathcal M([0,1])$. This reasoning works in general for distributed 
control on a compact subset $\omc$ of the domain $\Omega$.

\section{Optimality conditions}\label{sec:optimality}

We apply Fenchel duality to derive optimality conditions for minimizers of \eqref{eq:prob_reduced}. 
For the reader's convenience, we recall duality theory, e.g., from \cite[Chapter II.4]{Ekeland:1999a}. For a functional $\calF:W\to\Rbar:=\R\cup\{\infty\}$ defined on a Banach space $W$, let $\calF^*:W^*\to\Rbar$ denote the Fenchel conjugate of $\calF$ given for $w^*\in W^*$ by
\begin{equation}
    \calF^*(w^*) = \sup_{w\in W}\ \dual{w^*}{w}_{W^*,W} - \calF(w).
\end{equation}
Furthermore, let
\begin{equation}
    \partial\calF(w) := \set{w^*\in W^*: \dual{w^*}{\tilde w-w}_{W^*,W}\leq \calF(\tilde w)-\calF(w)\text{ for all }\tilde w\in W}
\end{equation}
denote the subdifferential of the convex function $\calF$ at $w$, which reduces to the G\^ateaux-derivative $\calF'(w)$ if it exists.
These definitions immediately yield the Fenchel--Young inequality
\begin{equation}\label{eq:fenchel_young}
    \dual{w^*}{w}_{W^*,W}\leq \calF^*(w^*)+\calF(w) \qquad\text{ for all }w\in W, w^*\in W^*,
\end{equation}
where equality holds if and only if $w^*\in\partial\calF(w)$.

The Fenchel duality theorem states that if $\calF:W\to\Rbar$ and $\calG:Z\to\Rbar$ are proper, convex, and lower semicontinuous functionals on the Banach spaces $X$ and $Z$, $\Lambda:W\to Z$ is a continuous linear operator, and there exists a $w_0\in W$ such that $\calF(w_0)<\infty$, $\calG(\Lambda w_0)<\infty$, and $\calG$ is continuous at $\Lambda w_0$ (a generalized Slater condition), then
\begin{equation}
    \label{eq:fenchel_duality}
    \inf_{w\in W} \calF(w)+\calG(\Lambda w) = \sup_{z^*\in Z^*}-\calF^*(\Lambda^* z^*) - \calG^*(-z^*),
\end{equation}
and the right-hand side of \eqref{eq:fenchel_duality} -- the \emph{dual problem} -- has at least one solution. Furthermore, the equality in \eqref{eq:fenchel_duality} is attained at $(\bar w, \bar z^*)\in W\times Z^*$ if and only if
\begin{equation}
    \label{eq:formal_opt}
    \left\{\begin{aligned}
            \Lambda^* \bar z^* &\in\partial\calF(\bar w),\\
            - \bar z^* &\in\partial\calG(\Lambda \bar w),
    \end{aligned}\right.
\end{equation}
holds; see, e.g., \cite[Remark III.4.2]{Ekeland:1999a}.

We wish to apply the Fenchel duality theorem to \eqref{eq:prob_reduced}, where
$\Lambda$ would take the role of the control-to-observation mapping $S$. Since $\M$
is non-reflexive, the dual problem would be posed in $\M^*$, which is difficult
to characterize. We therefore follow a \emph{pre-dual} approach as in \cite{Clason:2010a,Clason:2011a}, where we introduce the optimization problem
\begin{equation}\label{eq:predualpr}
    \inf_{h\in \dom \pred{S}} \frac12\norm{h+y_d}_{L^2(\omo)}^2-\frac12 \norm{y_d}_{L^2(\omo)}^2 + \delta_{C(\omc)^+}(\pred{S} h)
\tag{$\pred{}$P}
\end{equation}
(obtained by formal application of Fenchel duality) and show that its Fenchel dual coincides with problem \eqref{eq:prob_reduced}. 
\begin{remark}
    Before delving into a deeper analysis, let us point out that the pre-dual
    problem \eqref{eq:predualpr} is essentially a state-constrained optimal control problem with control
    $h \in \dom \pred{S} \subset L^2(\omo)$ and state $p:=\pred{S}h \in C(\overline \Omega)$, i.e.,
    \begin{equation}\label{eq:dualsc}
        \inf_{h\in {\dom \pred{S}}} \frac12\norm{h+y_d}_{L^2(\omo)}^2-\frac12 \norm{y_d}_{L^2(\omo)}^2 
        \quad\text{s.\,t.}\quad  \Aad p={\pred{E}}h, \quad  \pred{B} p \ge 0, \text{ on } \omc.   
    \end{equation}
    However, it has the slightly unusual characteristics that the state does
    not appear in the objective and that the inequality constraint is imposed on a subdomain.

    A further complication arises if $\dom \pred{S}$ is a proper subset of $L^2(\omo)$. 
    This case corresponds to a state-constrained problem where the control-to-state mapping does not map into the space of continuous functions. Such problems have been analysed in
    \cite{Schiela:2009}. The analysis performed in this section may offer an alternative approach to this class of problems.
\end{remark}

Problem \eqref{eq:predualpr} is strictly convex and admits a feasible point by \cref{ass:key} and thus is non-trivial, i.e., admits a finite infimum. If $\dom \pred{S}$ is not closed, we cannot expect \eqref{eq:predualpr} to
have a minimizer. However, any minimizing sequence is bounded in $L^2(\omo)$ and thus has a weak cluster point $\bar h \in L^2(\omo)$. 
In fact, by strict convexity of the term $\norm{h+y_d}_{L^2(\omo)}^2$, any minimizing sequence converges even strongly to the unique limit $\bar h$. 
While $\bar h$ is possibly not contained in $\dom \pred{S}$ -- and hence $\pred{S}\bar h$ is not defined -- we can express the limit using a suitable extension of $\pred{S}$ which we will define below. 

Although the Fenchel duality theorem is not directly applicable since $\pred{S}$ may be an unbounded operator, a modification of the arguments in \cite{Ekeland:1999a} shows that the statement still holds. In our argumentation, we can make use of the fact that we have already established existence of solutions of the dual problem in \cref{thm:existence}. For the sake of completeness, we give here the full proof, where we closely follow \cite[Chapter II.4]{Ekeland:1999a}.
Let us define for problem \eqref{eq:predualpr} the perturbation function $\Phi:L^2(\omo)\times C(\omc)\to\overline\R$ by 
\begin{equation}
    \Phi(h,v) := \frac12\norm{h+y_d}_{L^2(\omo)}^2-\frac12 \norm{y_d}_{L^2(\omo)}^2 + \delta_{C(\omc)^+}(\pred{S}h-v)+\delta_{\dom \pred{S}}(h).
\end{equation}
Clearly, $\Phi(h,v)$ is convex but -- by the last term -- not lower semicontinuous with respect to $h$ unless $\dom \pred{S}=L^2(\omo)$. 
Furthermore, $\inf_h\Phi(h,0)$ coincides with \eqref{eq:predualpr} and hence is finite. 

Consider now the Fenchel conjugate $\Phi^*:L^2(\omo)\times \M(\omc)\to\overline\R$ of $\Phi$ with respect to $(h,v)$. 
\begin{lemma}
    The dual problem
    \begin{equation}\label{eq:dualpr}
        \sup_{v^*\in\M(\omc)} -\Phi^*(0,v^*)
    \end{equation}
    coincides with problem \eqref{eq:prob_reduced}. Furthermore, if \cref{ass:key} is satisfied, the supremum is attained at $\bar v^* = \bar u$. 
\end{lemma}
\begin{proof}
    By definition, the Fenchel conjugate at $h^*=0$ is given by
    \begin{equation}
        \begin{aligned}[b]
            \Phi^*(0,v^*) &= \sup_{h\in \dom\pred{S},v\in C(\omc)}\  \dual{v^*}{v}_{\M,C} - \Phi(h,v) \\
                          &= \sup_{S^*h-v \in C(\omc)^+} \left(\dual{v^*}{v}_{\M,C} - \frac12\norm{h+y_d}_{L^2(\omo)}^2\right)+\frac12 \norm{y_d}_{L^2(\omo)}^2.
        \end{aligned}
    \end{equation}
    Using that $\dom \pred{S}$ is dense in $L^2(\omo)$ and introducing for $h\in\dom\pred{S}$ the function $p:= \pred{S}h - v\in C(\omc)$ then yields for the case that 
    $v^*\in \dom S$:
    \begin{equation}
        \begin{aligned}[b]
            \Phi^*(0,v^*) &= \sup_{h\in \dom\pred{S}, p \in C(\omc)^+} \left(\dual{v^*}{\pred{S}h-p}_{\M,C}  - \frac12\norm{h+y_d}^2_{L^2(\omo)}\right)+\frac12 \norm{y_d}_{L^2(\omo)}^2\\
                          &= \sup_{h\in \dom\pred{S}, p \in C(\omc)^+} \left(\dual{Sv^*}{h}_{L^2(\omo)}-\dual{v^*}{p}_{\M,C}  - \frac12\norm{h+y_d}^2_{L^2(\omo)}\right)+\frac12 \norm{y_d}_{L^2(\omo)}^2\\
                          &= \sup_{h\in \dom\pred{S}, p \in C(\omc)^+} \left(-\dual{v^*}{p}_{\M,C} -  \frac12\norm{h}_{L^2(\omo)}^2+\dual{h}{Sv^*-y_d}_{L^2(\omo)}\right).
        \end{aligned}
    \end{equation}
    If, in contrast, $v^*\notin\dom S$, there exists a sequence $\{h_n\}_{n\in\N}\subset\dom\pred{S}$, bounded in $L_2(\omo)$, such that $\dual{v^*}{\pred{S}h_n}_{\M,C}\to \infty$. Hence the first term in the first line is unbounded, while the opthers are bounded, and thus $\Phi^*(0,v^*)=\infty$. 
    We therefore assume that $v^*\in\dom S$ and maximize separately with respect to $p$ and $h$. Considering the first term, we have that $\dual{v^*}{p}_{\M,C} < 0$ for some $p\ge 0$ implies that $\Phi^*(0,v^*)=\infty$. Otherwise, the supremum is attained at $p = 0$ and is $0$. For the second term, we use that the functional is differentiable with respect to $h$ to deduce that the supremum is attained at $h=Sv^*-y_d$. Together, we obtain
    \begin{equation}
        \Phi^*(0,v^*) = \frac12 \norm{S v^*-y_d}_{L^2(\omo)}^2+\delta_{\M(\omc)^+}(v^*)+\delta_{\dom S}(v^*).
    \end{equation}

    Writing $u:=v^*$, we see that the dual problem \eqref{eq:dualpr} is precisely our original problem \eqref{eq:prob_reduced}, which by \cref{thm:existence} has a solution $\bar u\in\dom S\subset\M(\omc)$. 
\end{proof}

To derive optimality conditions, we first show that the duality gap between \eqref{eq:predualpr} and \eqref{eq:prob_reduced} is zero.
\begin{proposition}\label{pro:nogap}
    We have that
    \begin{equation}\label{eq:nogap}
        \inf_{h\in L^2(\omo)} \Phi(h,0) =  \sup_{v^*\in \M(\omc)} -\Phi^*(0,v^*).
    \end{equation} 
\end{proposition}
\begin{proof}
    The claim follows from \cite[Proposition III.2.1]{Ekeland:1999a} if Problem \eqref{eq:predualpr} is normal, i.e., the mapping $v\mapsto \inf_h \Phi(h,v)$ is lower semicontinuous at $0$. To verify this, it suffices to show that for each feasible point $h_v\in\dom\Phi(h,v)$, we can find a nearby feasible point $h_0\in\dom\Phi(h,0)$ with $\Phi(h_v,v)$ close to $\Phi(h_0,0)$. This can be achieved by adding a small multiple of the function $\pdS$ from \cref{ass:key}, since $\pred{S}\pdS$ is strictly positive and the perturbations are measured in the $C(\omc)$-norm. 

    Thus, for given $\varepsilon>0$  we can find $\delta>0$ such that with $\|v\|_{L^\infty(\omo)} < \delta$, $h_0:=h_v+\varepsilon h$ is feasible
    for the original problem, as long as $h_v$ is feasible for the perturbed problem. Moreover, it is easy to see that 
    $\Phi(h_0,0)-\Phi(h_v,v) \le \tau(\varepsilon)$ with $\tau \to 0$ as $\varepsilon \to 0$. Taking infima, this implies that
    \begin{equation}
        \inf_h \Phi(h,0) \le \inf_h \Phi(h,v)+\tau, 
    \end{equation}
    which in turn yields the desired lower semicontinuity and thus \eqref{eq:nogap}.
\end{proof}

To derive optimality conditions from the equality \eqref{eq:nogap}, we continue as in \cite[\S{}\,III, equation (4.22)]{Ekeland:1999a}. We first derive a limiting form of the optimality conditions.
\begin{proposition}\label{pro:opt_lim}
    Let $\{h_n\}_{n\in\N}\subset \dom \pred{S}\subset L^2(\omo)$ be a minimizing sequence for Problem \eqref{eq:predualpr} with $h_n\to \bar h\in L^2(\omo)$, and let $\bar u\in\M(\omc)$ be the solution to Problem \eqref{eq:dualpr}. 
    Then,
    \begin{equation}
        \label{eq:opt_lim}
        \left\{
            \begin{aligned}
                \bar h &= S\bar u -y_d,\\
                \pred{S}h_n &\geq 0,\quad \bar u \geq 0,\quad \lim_{n\to \infty}\dual{\bar u}{\pred{S} h_n}_{\M,C} = 0.
            \end{aligned}
        \right.
    \end{equation}
\end{proposition}
\begin{proof}
    By definition of $\Phi^*$, \cref{pro:nogap} implies that if $\{h_n\}_{n\in\N}$ is a minimizing sequence of $\Phi(\cdot,0)$ and $\bar u$ is a minimizer of $\Phi^*(0,\cdot)$, we have 
    \begin{equation}
        \lim_{n\to \infty}\Phi(h_n,0)+\Phi^*(0,\bar u) = 0. 
    \end{equation}
    We now use continuity of $\|\cdot\|_{L^2(\omo)}$ with respect to $h_n \to \bar h$ (recall that this limit exists due to the strict convexity of the first term in \eqref{eq:predualpr}), which yields
    \begin{equation}
        \begin{aligned}
            0 &=\lim_{n\to \infty} \Phi(h_n,0)+\Phi^*(0,\bar u)\\
              &= \frac12\norm{\bar h+y_d}_{L^2(\omo)}^2-\frac12 \norm{y_d}_{L^2(\omo)}^2 + \lim_{n\to \infty}\delta_{C(\omc)^+}(S^* h_n)\\
            \MoveEqLeft[-1]+\frac12 \|S \bar u-y_d\|^2_{L^2(\omo)}+\delta_{\M(\omc)^+}(\bar u).
        \end{aligned}
    \end{equation}
    Next, we observe that, since $\bar u \in \dom S$ and thus $S\bar u \in L^2(\omo)^*$, we have the convergence 
    \begin{equation}
        \lim_{n \to \infty} \langle \bar u, \pred{S}h_n\rangle_{\M,C} = \lim_{n \to \infty}\langle S\bar u,h_n\rangle_{L^2(\omo)} = \langle S\bar u,\bar h\rangle_{L^2(\omo)}.
    \end{equation}
    Hence, continuing our last computation, we obtain
    \begin{equation}
        \begin{aligned}
            0 &= \left[\frac12\norm{\bar h+y_d}_{L^2(\omo)}^2-\frac12 \norm{y_d}_{L^2(\omo)}^2+\frac12 \|S \bar u-y_d\|^2_{L^2(\omo)}-\dual{S\bar u}{\bar h}_{L^2(\omo)}\right]\\
            \MoveEqLeft[-1] +\Big[\lim_{n\to \infty}\delta_{C(\omc)^+}(S^* h_n)+\delta_{\M(\omc)^+}(\bar u)+\lim_{n \to \infty}\dual{\bar u}{\pred{S} h_n}_{\M,C}\Big].
        \end{aligned}
    \end{equation}
    We now argue that both brackets are non-negative. For the first bracket, we use the fact that the third term is the Fenchel conjugate of the sum of the first two terms to apply the Fenchel--Young inequality \eqref{eq:fenchel_young}. For the second bracket, feasibility of elements of a minimizing sequence (after passing to a subsequence if necessary) implies that $\pred{S}h_n \geq 0$ and $\bar u \geq 0$ and hence that the first two terms vanish. By definition of non-negativity of measures, positivity of $\bar u$ and $\pred{S}h_n$ implies 
    that $\dual{\bar u}{\pred{S} h_n}_{\M,C}\geq 0$ for all $n\in \N$ and hence that the third term is non-negative as well.
    Therefore, each bracket has to vanish separately. The first one immediately yields equality in \eqref{eq:fenchel_young} and hence that
    \begin{equation}
        \bar h \in \partial\left( \tfrac12 \|\cdot -y_d\|^2_{L^2(\omo)}\right)(S \bar u)= \{S \bar u-y_d\},
    \end{equation}
    i.e., the first relation of \eqref{eq:opt_lim}. From the second bracket, we directly obtain the remaining relations (i.e., the second line) of \eqref{eq:opt_lim}. 
\end{proof}

We now wish to pass to the limit $n\to\infty$ in \eqref{eq:opt_lim}, which is impeded by the fact that the operators $S$ and $\pred{S}$ are defined in the non-standard setting needed for measure-valued control. Recall that $\pred{S}$ -- which appears in $\dual{\bar u}{\pred{S} h_n}_{\M,C}$ -- is a restriction of its classical counter-part 
$\pred{S}_{H^1}:L^2(\omo)\to L^2(\omc)$. Hence, while
$\pred{S}\bar h$ may not be well-defined, $\pred{S}_{H^1}\bar h$ is well-defined since $\bar h \in L^2(\omo)$. Moreover, from $\bar u \in \dom S$ we can deduce not only that $\bar u \in \M(\omc)$ but also that $S\bar u \in L^2(\omo)$. 

We thus make use of $\pred{S}_{H^1}$ to define a new bilinear form 
\begin{equation}
    \dual{\cdot}{\cdot}_{\dom S,\ran\pred{S}_{H^1}} 
    : \dom S \times \ran \pred{S}_{H^1} \to \R
\end{equation}
that can be used as a replacement of the term $\dual{\bar u}{\pred{S} h_n}_{\M,C}$ in \eqref{eq:opt_lim} but is well-defined also for the limit $\bar h$. 
Let $u \in \dom S$ and $\lambda \in \ran \pred{S}_{H^1}$ with $h\in L^2(\omo)$ such that $\lambda = \pred{S}_{H^1} h$, then set
\begin{equation}
    \dual{u}{\lambda}_{\dom S,\ran\pred{S}_{H^1}} := \dual{Su}{h}_{L^2(\omo)}.
\end{equation}
With this definition, we obtain the following first-order necessary optimality conditions. 
\begin{theorem}\label{thm:optimality}
    Let $\bar u \in \M(\omc)$ be a minimizer of Problem \eqref{eq:problem}. Then there exist $\bar y \in W^{1,q'}(\Omega)$, $\bar p\in H^1(\Omega)$ 
    and $\bar \lambda \in \ran \pred{S}_{H^1} \subset  L^2(\omc)$ satisfying
    \begin{equation}\label{eq:optimality}
        \left\{\begin{aligned}
                {\pred{E}}({E}\bar y-y_d)-\Aad_{H^1} \bar p &=0,\\
                \bar \lambda - \pred{B} \bar p &=0,\\
                A\bar y- B\bar u &= 0,\\
                \bar\lambda \geq 0,\quad \bar u \geq 0,\quad 
                \dual{\bar u}{\bar \lambda}_{\dom S,\ran\pred{S}_{H^1}} 
                &= 0.
        \end{aligned}\right.\tag{OS}
    \end{equation}
\end{theorem}
\begin{proof}
    First, we note that $\dual{u}{\lambda}_{\dom S,\ran\pred{S}_{H^1}}$ is well-defined because $u \in \dom S$ implies $S u \in L^2(\omo)$, and because $h \in L^2(\omo)=\dom\pred{S}_{H^1}$. 
    We now to argue that this bilinear form can indeed be used in \eqref{eq:opt_lim}. For $h \in \dom \pred{S}$, we have $\lambda =\pred{S}_{H^1}h = \pred{S}h \in C(\omc)$ and thus 
    \begin{equation}
        \dual{u}{\lambda}_{\dom S,\ran\pred{S}_{H^1}} = \dual{Su}{h}_{L^2(\omo)}=\dual{u}{\pred{S}h}_{\M,C}=
        \dual{u}{\lambda}_{\M,C}.
    \end{equation}
    Furthermore, if $u\in \dom S$ and the sequence $\{h_n\}_{n\in\N}\subset\dom \pred{S}$ converges to $h$ in $L^2(\omo)$, 
    then
    \begin{equation}
        \begin{aligned}
            \lim_{n\to \infty}\dual{u}{\pred{S}h_n}_{\M,C}&=
            \lim_{n\to \infty}\dual{u}{\pred{S}_{H^1}h_n}_{\M,C}=
            \lim_{n\to \infty}\dual{u}{\pred{S}_{H^1}h_n}_{\dom S,\ran\pred{S}_{H^1}}\\
            &= \lim_{n\to \infty}\dual{Su}{h_n}_{L^2(\omo)}
            = \dual{Su}{h}_{L^2(\omo)}\\
            &= \dual{u}{\pred{S}_{H^1}h}_{\dom S,\ran\pred{S}_{H^1}}.
        \end{aligned}
    \end{equation}
    Thus, the limit $\lim_{n\to \infty}\dual{\bar u}{\pred{S} h_n}_{\M,C}$ in \eqref{eq:opt_lim} can be replaced 
    by $\dual{\bar u}{\pred{S}_{H^1}\bar h}_{\dom S,\ran\pred{S}_{H^1}}$ as claimed.

    Introducing the state $\bar y:= S\bar u=A^{-1}B\bar u$, an adjoint state $\bar p :=\Aad_{H^1}^{-1}\pred{E}\bar h=\Aad_{H^1}^{-1}\pred{E}(S\bar u-y_d)\in H^1(\Omega)$ and a Lagrangian multiplier $\bar \lambda := \pred{B} \bar p = \pred{S}_{H^1}\bar h\in \ran \pred{S}_{H^1}$ now yields \eqref{eq:optimality}.
\end{proof}

If $E$ is continuous, we can directly pass to the limit in the second relation of \eqref{eq:opt_lim} and obtain a Lagrange multiplier $\bar\lambda = \pred{S}\bar h\in C(\omc)$. 
\begin{corollary}\label{thm:optimality_cont}
    Assume that $E$ is continuous, and let $\bar u \in \M(\omc)$ be a minimizer of Problem \eqref{eq:problem}. Then there exist $\bar y \in \dom A$, $\bar p\in\dom \Aad\subset H^1(\Omega) \cap C(\overline \Omega)$, and $\bar \lambda \in  C(\omc)$ satisfying
    \begin{equation}\label{eq:optimality_cont}
        \left\{\begin{aligned}
                {\pred{E}}({E}\bar y-y_d)-\Aad \bar p &=0,\\
                \bar \lambda - \pred{B} \bar p &=0,\\
                A\bar y- B\bar u &= 0,\\
                \bar\lambda \geq 0,\quad \bar u \geq 0,\quad \dual{\bar u}{\bar\lambda}_{\M,C} &= 0.
        \end{aligned}\right.
    \end{equation}
\end{corollary}
In this case, the optimality conditions can also be obtained by direct application of the Fenchel duality theorem to problem \eqref{eq:predualpr}, where the last three relations of \eqref{eq:optimality_cont} are the complementarity conditions of the second relation of \eqref{eq:formal_opt}, which here read $-\bar u \in \partial\delta_{C^+}(\bar\lambda)$.

\section{Connection to problems with control costs}\label{sec:MY}

In this section, we show that problem \eqref{eq:prob_reduced} can be interpreted as the limit problem for vanishing $L^2$ or measure-space control costs.

\subsection{\texorpdfstring{$\scriptstyle L^2$}{L²} control costs}\label{sec:MY:L2}

We first connect the measure-space problem \eqref{eq:prob_reduced} with the classical control-constrained linear quadratic problem
\begin{equation}\label{eq:regproblem}
    \min_{u\in L^2(\omc)}  \frac{1}{2}\|Su-y_d\|_{L^2(\omo)}^2+\frac{\alpha}{2}\|u\|^2_{L^2(\omc)} + \delta_{L^2(\omc)^+}(u),
\tag{P$_\alpha$}
\end{equation}
which for every $\alpha>0$ is known to admit a minimizer $u_\alpha \in L^2(\omc)$; see, e.g., \cite[Theorem 2.14]{TroBook}. Arguing as in the proof of \cref{thm:existence}, it can be shown that $u_\alpha$ converges weakly-$*$ to some $\hat u$ in $\M(\omc)$ as $\alpha\to 0$ (up to a subsequence if $S$ is not injective). It is, however, not obvious that the limit
$\hat u$ coincides with the global minimizer $\bar u$ from \cref{thm:existence}. The validity of this assertion hinges on the question, whether there is a sequence 
$\{u_n\}_{n\in\N}\subset L^2(\omc)^+$ such that $u_n \rightharpoonup^* \bar u$ and $Su_n \rightharpoonup S\bar u$ in $L^2(\omo)$, i.e., whether optimal control
and optimal observation can be approximated simultaneously by a sequence of positive functions. 

Due to \cref{thm:wsd}, this is certainly the case if $E$ is continuous, since then $u_n \rightharpoonup^* \bar u$ implies $Su_n \to S\bar u$ by \cref{lem:wsd}.
\begin{theorem}
    Assume that $E$ is continuous, $S$ is injective, and $\omc$ is equipped with a measure $\nu$ such that $\nu(\omc \cap O)>0$ for every open set $O\subset \R^d$, such that
    $\omc \cap O$ is non-empty. Then 
    \begin{equation}
        u_\alpha \rightharpoonup^* \bar u \qquad\text{ and }\qquad Su_\alpha \to S\bar u. 
    \end{equation}
\end{theorem}
\begin{proof}
    By \cref{thm:wsd}, there exists a sequence $\{v_n\}_{n\in\N}\subset L^2(\omc)^+$ such that $v_n \rightharpoonup^* \bar u$. Since $E$ is continuous, this implies via \cref{lem:wsd}
    that $Sv_n \to S\bar u$ strongly and thus that $\|Sv_n-y_d\|_{L^2(\omo)} \to \|S\bar u-y_d\|_{L^2(\omo)}$. Denoting by $J_\alpha$ the functional in \eqref{eq:regproblem}
    and by $J$ the functional in \eqref{eq:prob_reduced}, we conclude that for each $\varepsilon > 0$ there are $v_n$ and $\alpha_n$ such that 
    \begin{equation} 
        J_{\alpha_n}(u_{\alpha_n}) = \inf_{u\in L^2(\omc)} J_{\alpha_n}(u) \le J_{\alpha_n}(v_n) \le J(\bar u)+\varepsilon.
    \end{equation}
    Hence, $\{u_{\alpha_n}\}_{n\in\N}$ is a minimizing sequence for $J$, which satisfies -- like any minimizing sequence -- the properties stated in the proof of \cref{thm:existence}.
    This yields our assertions. 
\end{proof}
On the other hand, if $E$ and thus $S$ is unbounded, the graph norm on $\dom S$, defined by $\|u\|_S := \|u\|_{\M(\omc)}+\|Su\|_{L^2(\omo)}$, is strictly stronger than $\|u\|_{\M(\omc)}$. Thus, 
there may be sequences in $L^2(\omc)$ that converge weakly-$*$ in $(\M(\omc),\|u\|_{\M(\omc)})$ but are unbounded in $(\dom S,\|u\|_{S})$ and thus cannot converge weakly-$*$ with respect to this norm. 
Hence if $S$ is unbounded, the weak-$*$ sequential closure of $L^2(\omc)$ may be a proper subset of $\dom S$, and thus we cannot expect in general that our global minimizer $\bar u$ can be approximated 
by a minimizing sequence in $L^2(\omc)$. 

\bigskip

Although the necessary optimality conditions for Problem \eqref{eq:regproblem} are standard (see, e.g., \cite[Theorem 2.22]{TroBook}), it is instructive to derive them using the convex analysis framework employed for \eqref{eq:prob_reduced}.
Since Problem \eqref{eq:regproblem} is posed in the Hilbert space $L^2(\omc)$ and we have assumed $E$ to be continuous, we can apply the Fenchel duality theorem directly, 
where we denote by $\mathcal{F}^*$ the tracking term and by $\mathcal{G}_\alpha^*$ the two remaining terms in \eqref{eq:regproblem}.
To derive an explicit characterization of the second relation of \eqref{eq:formal_opt}, we set $\lambda_\alpha := S^*h_\alpha\in L^2(\omo)$ and use the fact that due to the Hilbert space setting, $\mathcal{G}_\alpha$ coincides with the Moreau envelope of $\delta_{L^2(\omc)^+}$, i.e.,
\begin{equation}
    \mathcal{G}_\alpha(\lambda) = \left(\frac{\alpha}{2}\|\cdot\|^2_{L^2(\omc)} + \delta_{L^2(\omc)^+}\right)^*(\lambda) = \left(\delta_{L^2(\omc)^+}\right)_\alpha(\lambda):=\min_{w\in L^2(\omc)^+} \frac1{2\alpha}\|w-\lambda\|_{L^2(\omc)}^2,
\end{equation}
see, e.g., \cite[Proposition 13.12]{Bauschke}. Hence, $\partial\mathcal{G}_\alpha$ coincides with the Yoshida regularization of $\partial \delta_{L^2(\omc)^+}$, i.e.,
\begin{equation}
    \partial(\mathcal{G}_\alpha)(\lambda) = \left(\partial\delta_{L^2(\omc)^+}\right)_\alpha(\lambda) := \frac1\alpha \left(\lambda - \mathrm{prox}_{\alpha\delta_{L^2(\omc)^+}}(\lambda)\right)=
    \frac1\alpha \left(\lambda - \mathrm{proj}_{{L^2(\omc)^+}}(\lambda)\right),
\end{equation}
since the proximal mapping of an indicator function of a convex set $C$ is given by the metric projection onto $C$; see, e.g., \cite[Proposition 12.29]{Bauschke}. After some algebraic manipulations, we thus obtain the 
the optimality system 
\begin{equation}\label{eq:regopt}
    \left\{
        \begin{aligned}
            \lambda_\alpha &= S^*(Su_\alpha-y^d),\\
            u_\alpha &= \frac1\alpha \max\left(0,-\lambda_\alpha\right),
        \end{aligned}
\right.\tag{OS$_\alpha$}
\end{equation}
where $\max$ is to be understood pointwise almost everywhere in $\omc$.
Note that the system \eqref{eq:regopt} coincides with the well-known projection formulation of the optimality condition for the control-constrained linear-quadratic problem \eqref{eq:regproblem}; see, e.g., \cite[Theorem 2.28]{TroBook}.

\subsection{Measure-space control costs}\label{sec:MY:M}

We now connect problem \eqref{eq:prob_reduced} with the non-negative ``sparse control problem''
\begin{equation}\label{eq:m_problem}
    \min_{u\in \M(\omc)}  \frac{1}{2}\|Su-y_d\|_{L^2(\omo)}^2+\beta\|u\|_{\M(\omc)} + \delta_{\M(\omc)^+}(u)
\tag{P$_\beta$}
\end{equation}
considered in \cite{Clason:2011a}. Existence of an optimal control $u_\beta \in \M(\omc)^+$ can be shown as in \cref{thm:existence}, using the fact that a minimizing sequence is necessarily bounded in $\M(\omc)$ by virtue of the additional (weak-$*$ lower semi-continuous) term. Similarly, by the minimizing property of $u_\beta$, the family $\{Su_\beta\}_{\beta>0}$ is bounded in $L^2(\omo)$ and hence $u_\beta$ converges weakly-$*$ to $\bar u$ in $\M(\omc)$ as $\beta\to 0$ (up to a subsequence if $S$ is not injective) if \cref{ass:key} holds and $E$ is continuous. If on the other hand $E$ is unbounded, the discussion in \cref{sec:MY:L2} shows that $\dom S$ is in general not weakly-$*$ closed, and we cannot expect weak-$*$ convergence of $u_\beta$ to a minimizer $\bar u$.

Optimality conditions for \eqref{eq:m_problem} with a bounded control-to-observation mapping $S$ can be derived by application of the Fenchel duality theorem, making use of the fact that the Fenchel conjugate of 
\begin{equation}
    \calG_\beta:C(\omc) \to \Rbar,\qquad \calG_\beta(\lambda) = \delta_{\{v\geq -\beta\}}(\lambda) = \begin{cases} 0 & \lambda(x) \geq -\beta \text{ for all }x\in \omc,\\
        \infty & \text{else},
    \end{cases}
\end{equation}
is given by
\begin{equation}
    \calG_\beta^*:\M(\omc)\to\Rbar, \qquad \calG_\beta^*(u) = \beta\norm{u}_{\M(\omc)} + \delta_{\M(\omc)^-}(u),
\end{equation}
see \cite[Remark 2.5]{Clason:2011a}. (Recall that by \eqref{eq:fenchel_duality} the dual problem involves $\calG^*_\beta(-u)$.) 
Fenchel duality now leads to the necessary optimality conditions
\begin{equation}\label{eq:optimality_beta_cont}
    \left\{\begin{aligned}
            {\pred{E}}({E} y_\beta-y_d)-\Aad p_\beta &=0,\\
            \lambda_\beta - \pred{B} p_\beta &=0,\\
            Ay_\beta- Bu_\beta &= 0,\\
            \lambda_\beta \geq -\beta,\quad u_\beta \geq 0,\quad \dual{u_\beta}{\lambda_\beta+\beta}_{\M,C} &= 0,
    \end{aligned}\right.
\end{equation}
see again \cite[Remark 2.5]{Clason:2011a}, where the last relation was equivalently expressed as a variational inequality. Setting $\beta=0$, we recover \eqref{eq:opt_lim}. 

The optimality conditions \eqref{eq:optimality_beta_cont} are frequently used as a justification for calling $u_\beta$ a \emph{sparse control}: From the last relations, we see that $u_\beta$ must be zero on all subsets of $\omc$ where $\pred{B}p_\beta$ is strictly greater than $-\beta$. Hence, the support of $u_\beta$ is contained in the set $\{x\in\omc:\pred{B}p_\beta(x) = -\beta\}$, which in many situation (e.g., if $p_\beta$ is harmonic) can be argued to be a set of zero Lebesgue measure. Furthermore, increasing $\beta$ will decrease the size of this set. The same argument is possible for \eqref{eq:opt_lim}: the optimal control $\bar u$ must be zero on all subsets with $\pred{B}p_\beta>0$, and hence the support of $\bar u$ is contained in $\{x\in\omc:\pred{B}\bar p(x) = 0\}$ (which has Lebesgue measure zero in similar situations as in the case $\beta>0$). This implies that optimal measure-space controls have an inherent sparsity independent of the sparsity-promoting control cost, whose role is solely to control the size of the support.

We can also apply our framework from \cref{sec:optimality} to derive optimality conditions for unbounded observation operators (which cannot be treated using the standard approach as in, e.g., \cite{Clason:2011a}). Proceeding exactly as before with $\delta_{C(\omc)^+}$ replaced by $\delta_{\{v\geq -\beta\}}$ and $\delta_{\M(\omc)^+}$ replaced by $\beta\norm{{\cdot}}_{\M(\omc)}+\delta_{\M(\omc)^+}$, we obtain the modified optimality conditions
\begin{equation}\label{eq:optimality_beta}
    \left\{\begin{aligned}
            {\pred{E}}({E} y_\beta-y_d)-\Aad_{H^1}  p_\beta &=0,\\
            \lambda_\beta - \pred{B}  p_\beta &=0,\\
            A y_\beta- B u_\beta &= 0,\\
            \lambda_\beta \geq -\beta,\quad  u_\beta \geq 0,\quad 
            \beta\norm{ u_\beta}_{\M(\omc)}+\dual{ u_\beta}{ \lambda_\beta}_{\dom S,\ran\pred{S}_{H^1}} 
            &= 0.
\end{aligned}\right.\tag{OS$_\beta$}
\end{equation}
Again setting $\beta=0$, we recover \eqref{eq:optimality}. However, since the last relation can no longer be interpreted pointwise, a sparsity property of $u_\beta$ does not follow directly.

\section{Numerical solution}\label{sec:discretization}

The numerical solution is based on the conforming discretization of $\M(\omc)$ introduced in \cite{Clason:2012}, which we briefly recall. The starting point is to replace $S:\M(\omc)\to L^2(\omo)$ by its finite element semidiscretization $S_h:\M(\omc)\to Y_h$, where $Y_h\subset L^2(\omo)$ is a finite-dimensional space spanned by the usual continuous piecewise linear nodal basis (``hat'') functions attached to the vertices $\{x_j\}_{j=1}^{N}$ of a triangulation of $\overline\Omega$. We then consider the semidiscrete optimal control problem
\begin{equation}\label{eq:discrete:semi_problem}
    \min_{u \in \M(\omc)}\ \frac{1}{2}\|S_hu - y_d\|^2_{L^2(\omo)}+\delta_{\M(\omc)^+}(u).
\tag{P$_h$}
\end{equation}
Existence of an optimal control $\bar u$ can be shown as in \cref{sec:existence:minimizer}. Although the optimal state $\bar y_h = S_h\bar u$ is unique, this is no longer the case for the control due to the finite number of observations. However, there is a unique $\bar u_h\in\M(\omc)$ with $\bar y_h=S_h(\bar u_h)$ that can be represented as a linear combination of Dirac measures concentrated on the vertices $x_j$ contained in $\omc$; see \cite[Theorem~3.2]{Clason:2012}. We can thus restrict the minimization in \eqref{eq:discrete:semi_problem} over the set $U_h$ of such linear combinations. 
In this sense, this approach is related to a discretization method introduced in \cite{Winther:1978} for unconstrained linear-quadratic problems and also to the variational discretization of control-constrained problems of \cite{Hinze2005}.

This allows expressing Problem \eqref{eq:discrete:semi_problem} purely in terms of the expansion coefficients $\vec u$ of $\bar u_h$ and $\vec y$ of $\bar y_h$. Using that $\bar u_h\in\M(\omc)^+$ if and only if $\vec u \geq 0$ componentwise and applying the Fenchel duality theorem as in \cref{thm:optimality_cont} (all finite-dimensional operators being bounded) yields the fully discrete optimality conditions
\begin{equation}
    \label{eq:discrete:opt}
    \left\{\begin{aligned}
            A_h^T \vec p &= M_h(\vec y - \vec y_d),\\
            A_h \vec y &= B_h \vec u,\\
            -\vec u &\in \partial\delta_{(\R^N)^+}(B_h^T \vec p),
    \end{aligned}\right.
\tag{OS$_h$}
\end{equation}
where $A_h$ denotes the stiffness matrix corresponding to the differential operator $A$, $M_h$ the restricted mass matrix on the observation domain $\omo$, and $B_h^T$ the discrete restriction operator to the components of $\vec p$ corresponding to vertices contained in $\omc$. (Note the lack of mass matrix for the discrete state equation.) Since $\R^N$ is a Hilbert space, we can reformulate the last relation in \eqref{eq:discrete:opt} using resolvent calculus similarly as in \cref{sec:MY} as
\begin{equation}
    \label{eq:discrete:opt_ref}
    -\vec u = \frac1\alpha\left(B_h^T\vec p-\alpha \vec u - \mathrm{proj}_{(\R^N)^+}(B_h^T\vec p-\alpha \vec u)\right)
\end{equation}
for any $\alpha>0$; see also \cite[Theorem 4.41]{Kunisch:2008a}. (Comparing this relation with the last relation in \eqref{eq:regopt}, we remark that the only difference is the presence of $\vec u$ on the right-hand side.) 
In particular, for $\alpha=1$ we obtain
\begin{equation}
    \vec u = \max\left(0,\vec u - B_h^T \vec p\right),
\end{equation}
where the $\max$ is to be understood componentwise. 

It is well-known that the $\max$ operator is semismooth on $\R^N$ with Newton derivative at $\vec v$ in direction $\vec h$ is given componentwise by
\begin{equation}
    \left[D_N \max(0,\vec v)\vec h\right]_j = 
    \begin{cases} 
        h_j & \text{if } v_j >0,\\
        0 & \text{if } v_j \leq 0,
    \end{cases}
\end{equation}
and that system \eqref{eq:discrete:opt} therefore can be solved by a superlinearly convergent semismooth Newton method; see \cite{Kunisch:2008a,Ulbrich:2002a}.
To account for the local convergence of Newton methods, we compute a starting point by solving a sequence of discrete regularized problems analogous to \cref{sec:MY}.
Specifically, we add for $\alpha>0$ the $\ell_2$ penalty  $\frac\alpha2 |\vec u|_2^2$ and proceed as in \cref{sec:MY} to obtain
\begin{equation}
    \label{eq:discrete:opt_}
    \left\{\begin{aligned}
            A_h^T \vec p &= M_h(\vec y - \vec y_d),\\
            A_h \vec y &= B_h \vec u,\\
            \vec u &= \frac1\alpha\max\left(0,-B_h^T \vec p\right).
    \end{aligned}\right.
\tag{OS$_{h,\alpha}$}
\end{equation}
Since the last relation is explicit, we can eliminate $\vec u$ and apply a semismooth Newton method to the reduced system, starting with $\alpha=1$ and successively reducing $\alpha$, taking for each $\alpha$ the previous solution as starting point.

\section{Numerical examples}\label{sec:examples}

We illustrate the nature of the generalized measure-space controls with numerical examples for the Laplace equation on the unit square with homogeneous Dirichlet conditions, i.e., we take $\Omega=[-1,1]^2\subset\R^2$ and $A= -\Delta$. The domain is discretized using the standard uniform triangulation arising from $256\times 256$ equidistributed nodes. The optimal controls for the discretized problem are computed using a \textsc{matlab} implementation of the approach described in \cref{sec:discretization}, which can be downloaded from \url{https://github.com/clason/positivecontrol}.

\begin{figure}[t]
    \centering
    \begin{subfigure}{0.475\textwidth}
        \centering
        \includegraphics[width=\textwidth]{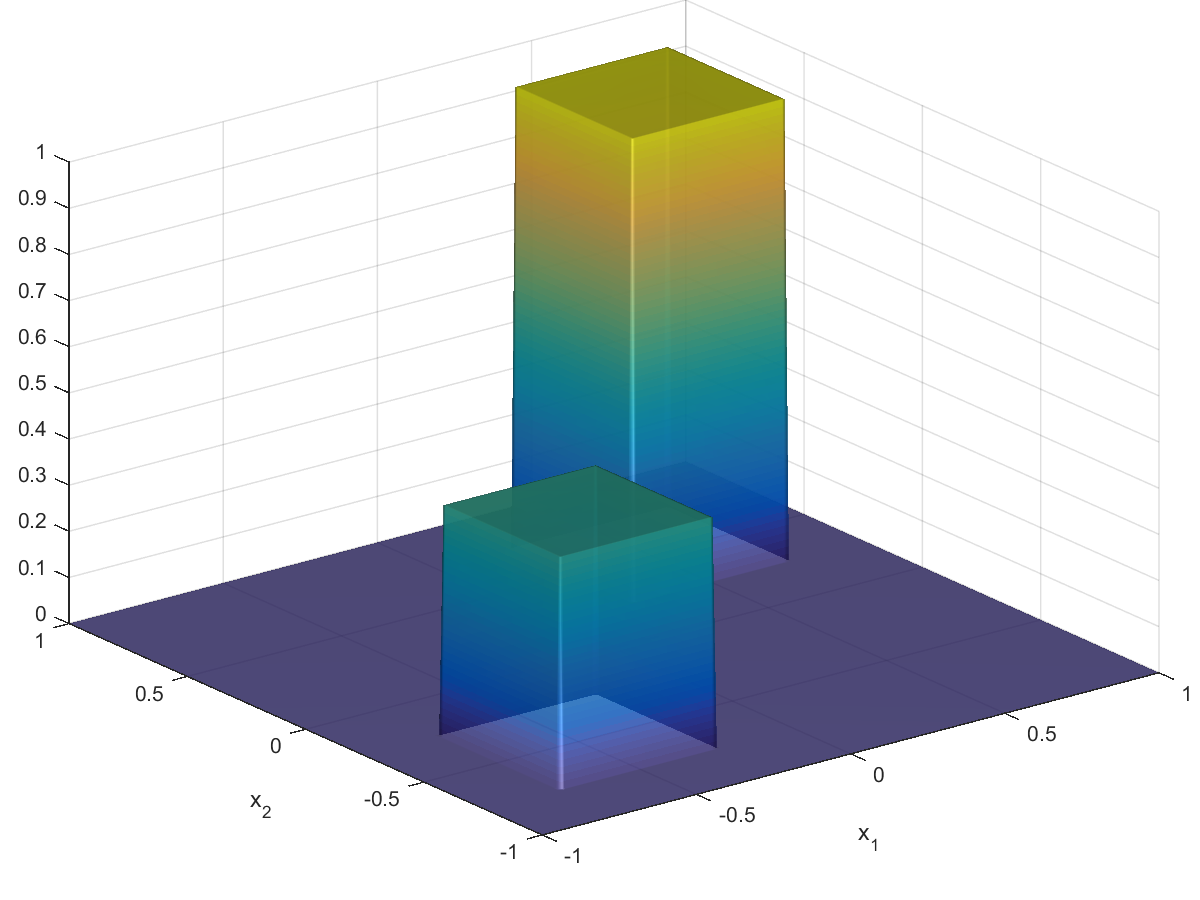}
        \caption{desired state $y_d$\label{fig:target:1}}
    \end{subfigure}
    \hfill
    \begin{subfigure}{0.475\textwidth}
        \centering
        \includegraphics[width=\textwidth]{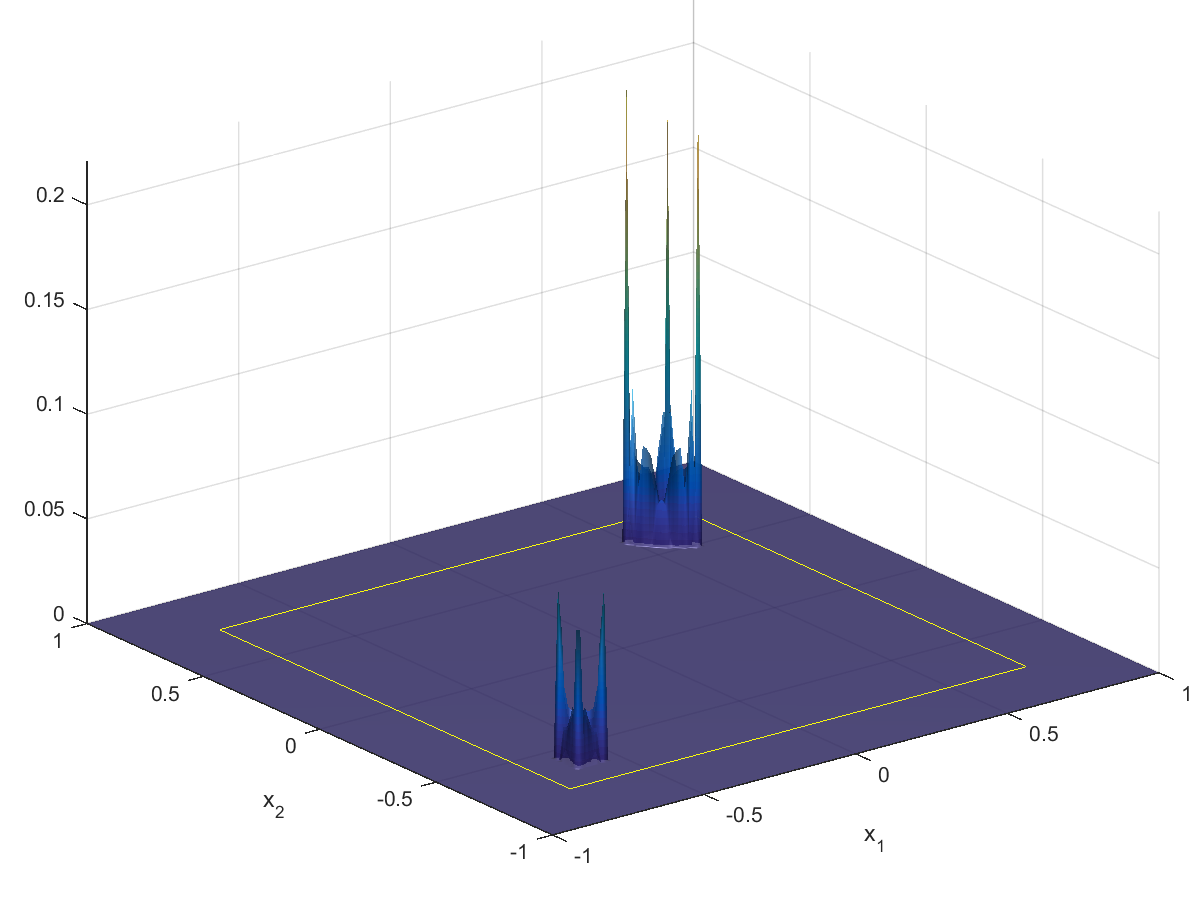}
        \caption{optimal control $\bar u_h$\label{fig:box:control}}%
    \end{subfigure}
    \newline
    \begin{subfigure}{0.475\textwidth}
        \centering
        \includegraphics[width=\textwidth]{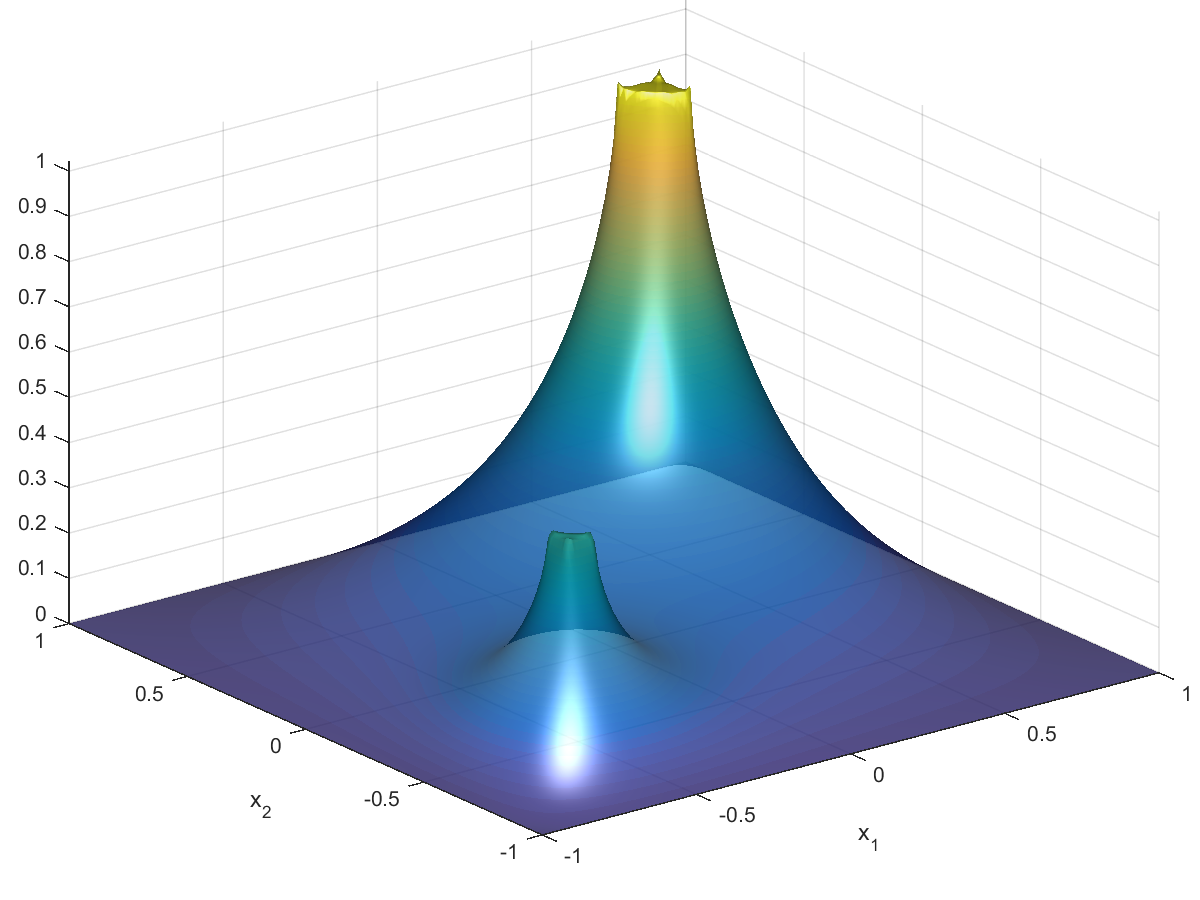}
        \caption{optimal state $\bar y_h$\label{fig:opt:1}}%
    \end{subfigure}
    \hfill
    \begin{subfigure}{0.475\textwidth}
        \input{box_sparsity.tikz} 
        \caption{sparsity pattern \label{fig:box:sparsity}}%
    \end{subfigure}
    \caption{First desired state and resulting optimal control and state}
    \label{fig:box}
\end{figure}
For the first example, we choose the desired state as
\begin{equation}
    \label{eq:target1}
    \begin{aligned}[t]
        y_d(x_1,x_2) &= \chi_{\{t:|t-0.5|<0.25\}}(x_1)\chi_{\{t:|t-0.5|<0.25\}}(x_2)\\
        \MoveEqLeft[-1] +
        \frac12\chi_{\{t:|t+0.5|<0.25\}}(x_1)\chi_{\{t:|t+0.5|<0.25\}}(x_2),
    \end{aligned}
\end{equation}
see \cref{fig:target:1}. According to the discussion at the end of \cref{sec:existence:examples}, the control domain has to be chosen as a proper subset of $\Omega$ for Problem \eqref{eq:problem} to be well-posed;
here we set
\begin{equation}
    \omc = \set{x\in\Omega:|x|_\infty\leq \tfrac34}.
\end{equation}
The observation domain is chosen as $\omo=\Omega$.
The expansion coefficients of the optimal control are shown in \cref{fig:box:control}, where the boundary of the control domain is also marked by a yellow line; the corresponding optimal state is shown in \cref{fig:opt:1}. It can be observed that the optimal control is sparse, which is in accordance with the discussion in \cref{sec:MY:M}. This is further illustrated in \cref{fig:box:sparsity} where the nodes with non-zero control coefficients are marked with a small circle.
\begin{figure}[t]
    \centering
    \begin{subfigure}{0.475\textwidth}
        \centering
        \includegraphics[width=\textwidth]{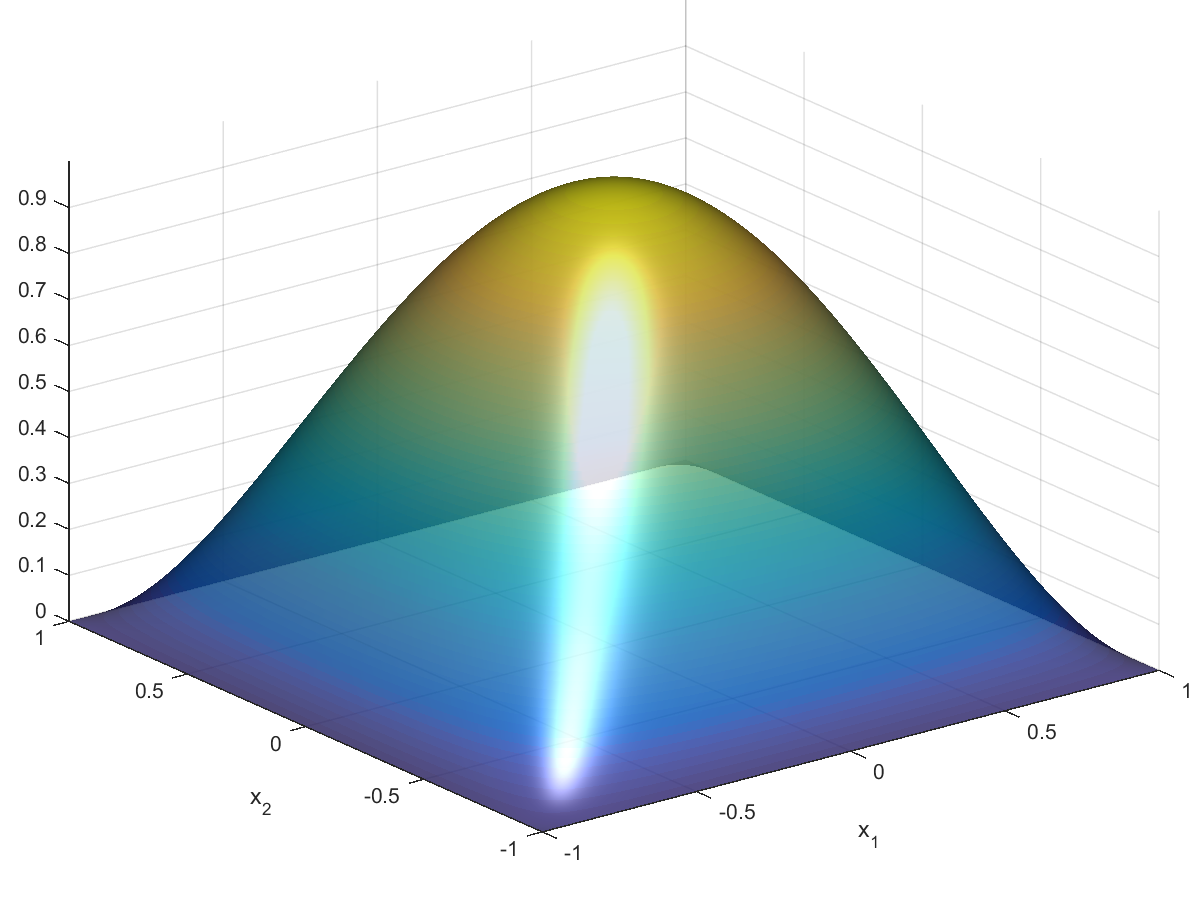}
        \caption{desired state $y_d$\label{fig:target:2}}%
    \end{subfigure}
    \hfill
    \begin{subfigure}{0.475\textwidth}
        \centering
        \includegraphics[width=\textwidth]{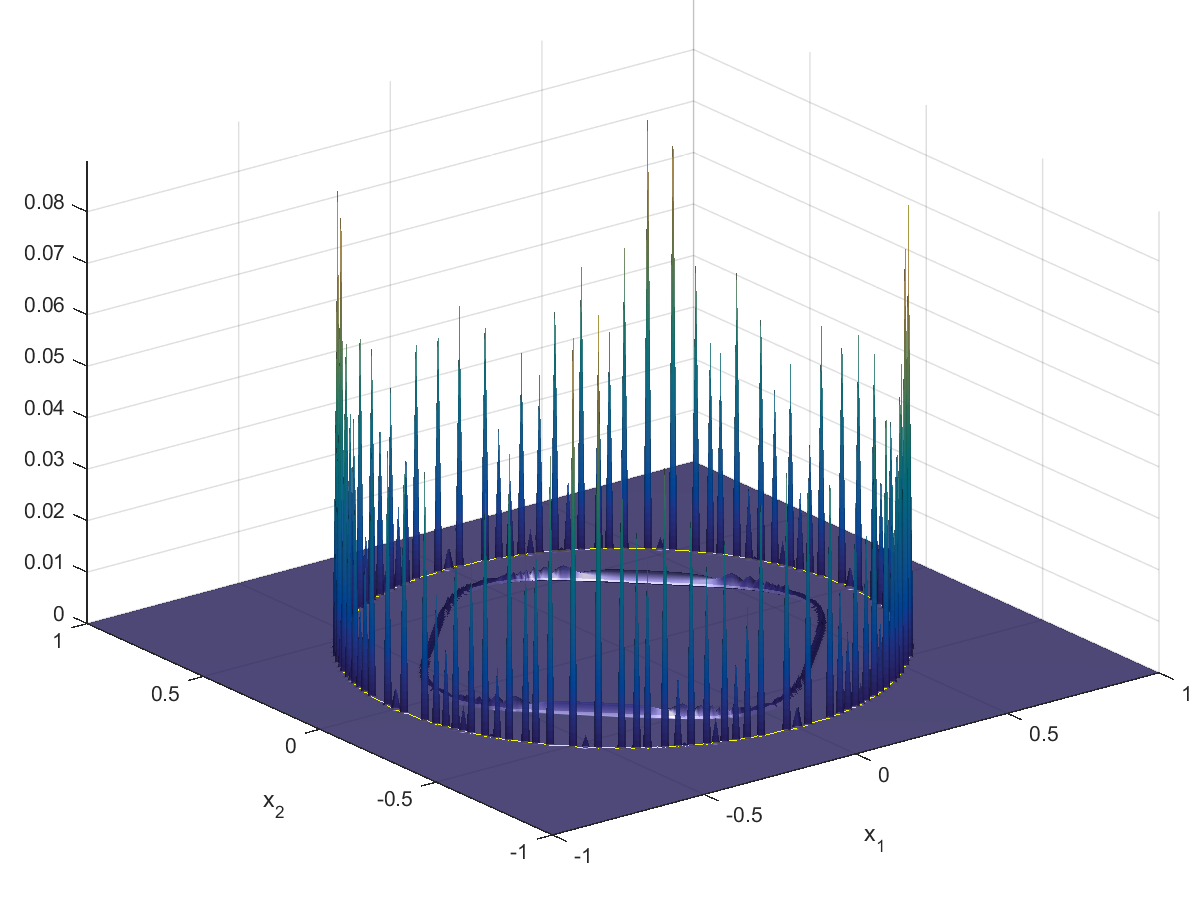}
        \caption{optimal control $\bar u_h$\label{fig:quad:control}}%
    \end{subfigure}
    \newline
    \begin{subfigure}{0.475\textwidth}
        \centering
        \includegraphics[width=\textwidth]{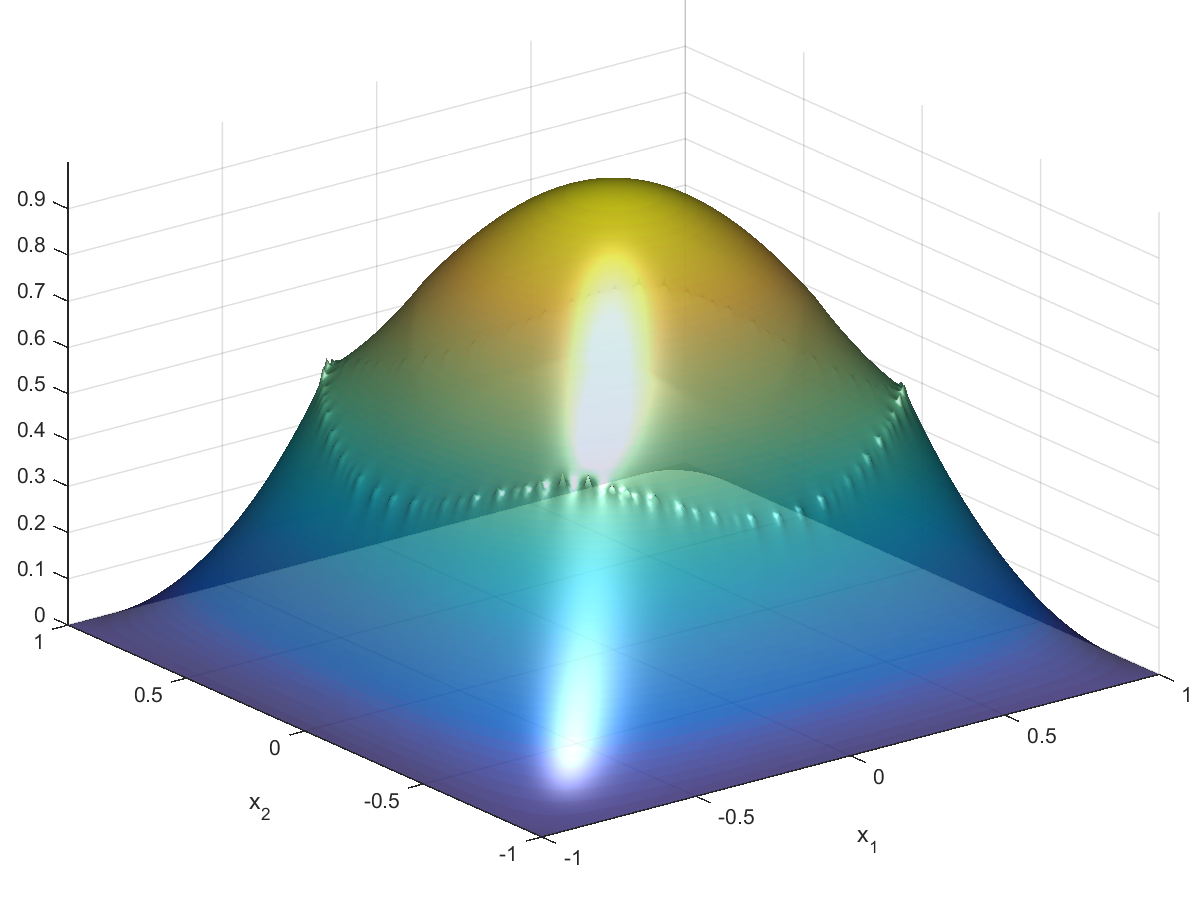}
        \caption{optimal state $\bar y_h$\label{fig:opt:2}}%
    \end{subfigure}
    \hfill
    \begin{subfigure}{0.475\textwidth}
        \centering
        \includegraphics[width=\textwidth]{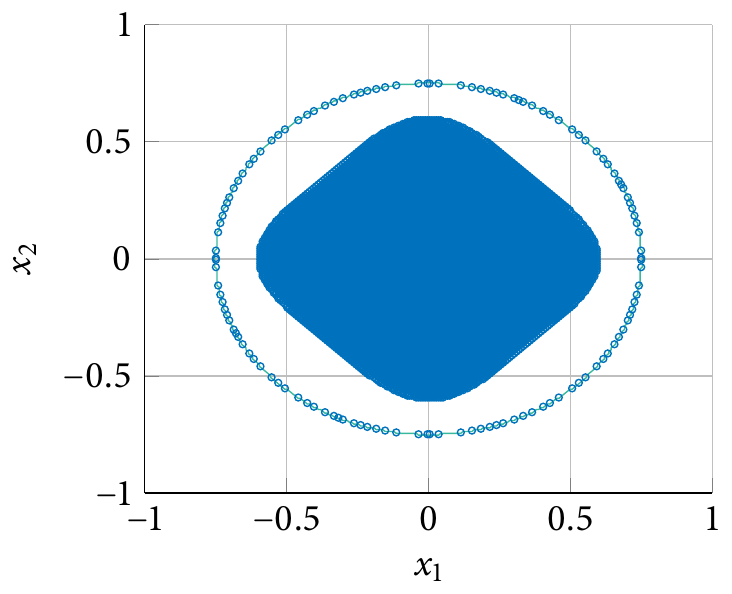} 
        \caption{sparsity pattern \label{fig:quad:sparsity}}%
    \end{subfigure}
    \caption{Second desired state and resulting optimal control and state}
    \label{fig:quad}
\end{figure}
The situation is different if the adjoint state satisfies $\bar p=0$ on an open set, which can happen if the desired state is (locally) attainable. We demonstrate this using 
\begin{equation}
    \label{eq:target2}
    y_d(x_1,x_2) = (1-x_1^2)(1-x_2^2),
\end{equation}
see \cref{fig:target:2}. The control domain is set to
\begin{equation}
    \omc = \set{x\in\Omega:|x|_2\leq \tfrac34},
\end{equation}
while the observation domain is again chosen as $\omo=\Omega$.
The corresponding optimal control $\bar u_h$, control domain, optimal state $\bar y_h$ and locations of non-zero components of $\vec u$ are shown in \cref{fig:quad:control}--\subref*{fig:quad:sparsity}. Here, the control consists of the sum of a line measure concentrated on the boundary $\partial\omc$ of the control domain and a distributed function (of small magnitude compared to the line measures in \subref*{fig:quad:control}) in the interior of $\omc$. (We remark that a similar behavior can be observed in the case of $\M$-norm penalties for attainable targets.)

We close this section with an example of Neumann boundary control, i.e., $\omc = \partial\Omega$, for the operator $A=-\Delta + c_0\,\mathrm{Id}$ with $c_0= 10^{-2}$, desired state \eqref{eq:target1} and $\omo = \Omega$. 
The corresponding optimal control $\bar u_h$ is shown in \cref{fig:neumann}, where we again observe a sparse solution.
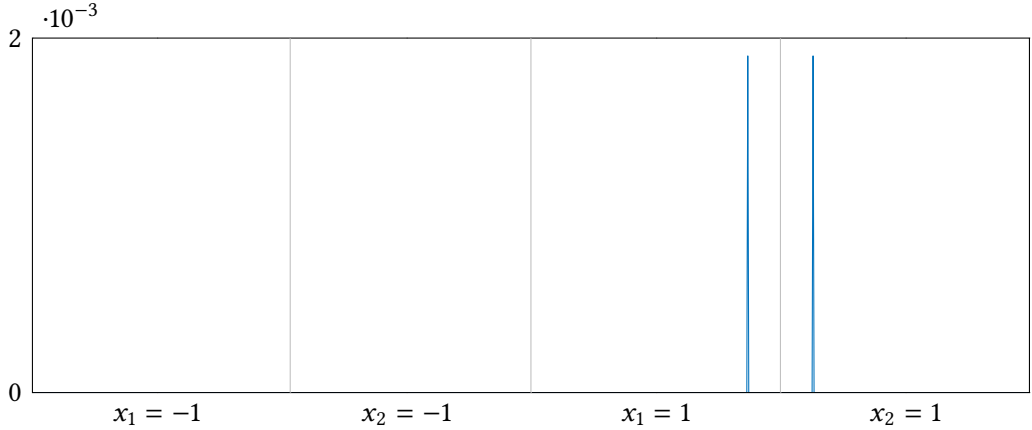
\begin{figure}[t]
    \centering
    \input{neumann_control.tikz}
    \caption{Optimal Neumann boundary control $\bar u_h$ for desired state in \cref{fig:target:1}}
    \label{fig:neumann}
\end{figure}

\section{Conclusion}

Optimal control problems with non-negativity constraints are coercive even without control costs, albeit only in the space of Radon measures. Existence of a strictly positive solution of the pre-adjoint equation verifies a pre-dual Slater condition, which yields existence of and optimality conditions for a minimizing control.
These results confirm the previously only numerically observed stability of the non-negative sparse control problem in \cite{Clason:2011a} as $\alpha\to 0$.
This approach is also applicable if the control-to-observation mapping is not continuous, using Fenchel duality for an unbounded operator. A conforming discretization of the space of Radon measures yields a discrete measure-space control problem that is amenable to the efficient numerical solution by a semismooth Newton method. The numerical examples demonstrate that optimal measure-space controls have an inherent sparsity property which does not require the presence of sparsity-promoting penalties. Rather, the measure-space setting allows the minimizing sequence to concentrate on lower-dimensional manifolds, which is prevented by $L^p$ control costs enforcing higher regularity. (Of course, an additional sparsity penalty can lead to even smaller support of the optimal control.)
This is another illustration of the fact that optimization problems in function spaces have a much more delicate structure than their finite-dimensional counterparts due to the richer topological properties of infinite-dimensional spaces.

This work can be extended in several directions. Although outside the scope of the current paper, an analysis of the conforming finite element discretization --  including convergence rates -- along the lines of \cite{Clason:2012} is certainly possible. One could also apply techniques developed for the Moreau--Yosida regularization of state constraints to obtain convergence rates for the regularization in \cref{sec:MY}. Finally, it would be worthwhile to investigate whether well-posedness in weak spaces can also hold for nonlinear or time-dependent problems without control costs.

\appendix

\section{Density in the cone of positive measures}\label{sec:appendix}

In this appendix we will prove \cref{thm:wsd}, needed in \cref{sec:existence:minimizer} and \cref{sec:MY}. 
We will need some notation. 
For a normed space $X$, let $B_X$ its closed unit ball, and for $S \subset X$, define its polar set
\begin{equation}
    S^\circ := \set{ x^* \in X^* : \langle x^*,x\rangle_{X^*,X} \le 1 \;
    \text{ for all } x\in S}.
\end{equation}
By switching the roles of $X$ and $X^*$ one defines the polar of a subset of $X^*$. 
Basic results on polar sets can be found in, e.g., \cite[\S{}\,IV.1]{Schaefer}.
We also need the following density result.

\begin{lemma}\label{lem:polar}
    Let $X$ be a separable Banach space, $U$ a linear subspace of $X^*$, and $S\subset X$.
    Assume that
    \begin{equation}
        \overline{co}\, (B_X \cup S) \supset (B_U \cap S^\circ)^\circ.
    \end{equation}
    Then $B_U \cap S^\circ$ is weakly-$*$ sequentially dense in $B_{X^*}\cap S^\circ$, i.e., for all $f\in X^*$ there exists a sequence $\{f_n\}_{n\in\N}\subset (B_U\cap S^\circ)$ such that $\langle f_n,x \rangle_{X^*,X}
    \to \langle f,x \rangle_{X^*,X}$ for all $x\in X$.
\end{lemma}
\begin{proof}
    We compute
    \begin{equation}
        \begin{aligned}
            B_{X^*} \cap S^\circ &= B_{X}^\circ \cap S^\circ = (B_X \cup S)^\circ
            =(\overline{co}\, (B_X \cup S))^\circ\\ &\subset (B_U \cap
            S^\circ)^{\circ\circ} = \overline{co}^{\sigma(X^*,X)}
            (B_U \cap S^\circ \cup \{0\}) = \overline{B_U \cap
            S^\circ}^{\sigma(X^*,X)}.
        \end{aligned}
    \end{equation}
    Here we used the following rules of polar calculus: $B_{X^*}=B_X^\circ$, $M^\circ \cap N^\circ = (M\cup N)^\circ$, $M^\circ = (\overline{co}\, M)^\circ$, $M\supset N \Rightarrow M^\circ \subset N^\circ$, and the bi-polar theorem $M^{\circ\circ}=\overline{co}\,(M\cup \{0\})$. 
    The last equality follows from $0 \in B_U \cap S^\circ$ and from convexity of $B_U \cap S^\circ$.  
    Now equality follows from $B_U\subset B_{X^*}$ and the fact that $B_{X^*} \cap S^\circ$ is closed in the $\sigma(X^*,X)$ topology since polar sets are always weakly-$*$ closed.

    Hence, the bounded set $B_U \cap S^\circ$ is weakly-$*$ dense in the bounded set $B_{X^*} \cap S^\circ$. 
    By separability of $X$, this implies weak-$*$ sequential density because $B_{X^*}$ is metrizable; see, e.g., \cite[Corollary~3.30]{Brezis:2010a}.
\end{proof}

\begin{lemma}\label{lem:cone}
    If $S \subset X$ is a cone, then
    \begin{equation}
        S^\circ =  \set{ x^* \in X^* : \langle x^*,x\rangle_{X^*,X} \le 0 \;
        \text{ for all } x\in S}
    \end{equation}
    and
    \begin{equation}
        \overline{co} \,(B_X \cup S) \supset B_X+S:= \set{ \phi\in X : \phi=\phi_1+\phi_2 : \phi_1\in B_X, \phi_2 \in S}.
    \end{equation}
    In particular, the conclusions of \cref{lem:polar} hold if $B_X + S \supset (B_U \cap S^\circ)^\circ$.
\end{lemma}
\begin{proof}
    For the first assertion, we note that if $\langle x^*,x\rangle_{X^*,X}>0$, then there exists an $\alpha > 0$ such 
    that $\langle \alpha x,x^*\rangle_{X,X^*}>1$.
    For the second assertion, observe that for $\rho\in[0,1)$ we have that
    \begin{equation}
        \rho (\phi_1+\phi_2) = \rho \phi_1+(1-\rho) (\rho/(1-\rho) \phi_2) \in co (B_X \cup S).
    \end{equation}
    Hence, letting $\rho \to 1$, we obtain  $ \phi_1+\phi_2 \in \overline{co}\, (B_X \cup S)$.
\end{proof}

\begin{theorem}\label{thm:wsd}
    Let $ Q $ be a compact subset of $\R^d$, equipped with a positive measure $\nu$ such that $\nu(\omega) > 0$ for each relatively open, non-empty subset $\omega\subset Q$.
    \begin{itemize}
        \item[(i)] Let $\mu \in \M( Q )^+$ be a positive measure. Then there exists a sequence of positive functions $0 \le f_n$ in $L^\infty( Q )$ with $f_n\rightharpoonup^\ast\mu$  
            and  $\|f_n\|_{L^1( Q )}\le \|\mu\|_{\mathcal M( Q )}$.
        \item[(ii)] Let $\mu \in \M( Q )$ be a signed measure. Then there exists a sequence $f_n$ in $L^\infty( Q )$ with $f_n\rightharpoonup^\ast\mu$  and  $\|f_n\|_{L^1( Q )}\le \|\mu\|_{\mathcal M( Q )}$.
    \end{itemize}
\end{theorem}
\begin{proof}
    We first consider assertion \textit{(i)} and note that
    for $S=\set{\phi\in C( Q ) : \phi\le 0}$, we have that
    \begin{equation}
        S^\circ=\set{\mu \in \mathcal M( Q ) : \dual{\mu}{\phi}_{\M(Q),C(Q)} \le 0  \;\text{ for all } \phi\in S}
        = \M( Q )^+.
    \end{equation}
    This follows by the fact that $S$ is a cone and from the Riesz representation theorem for positive linear functionals. 
    Let further $X=C( Q )$ and $U=L^\infty( Q )$, which can be interpreted as a subspace of $X^*=\mathcal M( Q )$.
    Application of \cref{lem:polar} will then yield our result (possibly after scaling of $\mu$ to $\|\mu\|_{\M(Q)}=1$).
    Thus, we have to check the condition $\overline{co}\, (B_X \cup S) \supset (B_U \cap S^\circ)^\circ$.

    By \cref{lem:cone}, we merely have to show $B_X+S \supset (B_U \cap S^\circ)^\circ$.
    Let $\phi\in C( Q )\setminus(B_X+S)$. 
    Since $\max\{\phi,0\}$ and $\min\{\phi,0\}$ are also continuous functions, this implies that there exists $x\in  Q $ such that $\phi(x) > 1$. 
    We now show that $\phi \notin (B_U \cap S^\circ)^\circ$, i.e., there exists $f \in B_U \cap S^\circ$ such that $\dual{f}{\phi}_{\M(Q),C(Q)} > 1$.
    Let $\alpha := \phi(x)-1>0$. 
    Then, by continuity of $\phi$, there exists an open neighborhood $\omega$ of $x$ such that $\phi|_{\omega} \ge 1+\alpha/2$. By assumption, $\nu(\omega)\not=0$.
    Set $f := \nu(\omega)^{-1}\chi_{\omega}$ (i.e., a scaled characteristic function), which yields $\|f\|_{L^1( Q )} = 1$ and $f \ge 0$.
    Thus, $f\in B_U \cap S^\circ$, and
    \begin{equation}
        \dual{f}{\phi}_{\M(Q),C(Q)} = \nu(\omega)^{-1}\int_\omega \phi\, d\nu \ge 1+\alpha/2>1.
    \end{equation}
    This shows that $\phi \not \in (B_U \cap S^\circ)^\circ$, which allows application of \cref{lem:polar}. 

    Assertion \textit{(ii)} now follows from assertion \textit{(i)} by splitting $\mu$ into a positive and negative part and approximating these separately via \textit{(i)} by positive and negative functions, respectively. 
\end{proof}

\section*{Acknowledgments}

This work was supported in part by the Austrian Science Fund (FWF) under grant SFB {F}32 (SFB ``Mathematical Optimization and Applications in Biomedical Sciences'').
Part of the work was carried out while the second author was interim professor at the University of Hamburg. 

\printbibliography
\end{document}

%% file: box_sparsity.tikz
\definecolor{mycolor1}{rgb}{0.00000,0.44700,0.74100}%
\definecolor{mycolor2}{rgb}{0.21783,0.72504,0.61926}%
\begin{tikzpicture}

\begin{axis}[%
    width=\textwidth,
    xmin=-1,
    xmax=1,
    xlabel={$x_1$},
    ymin=-1,
    ymax=1,
    ylabel={$x_2$},
    axis x line*=bottom,
    axis y line*=left,
    grid=major
]
\addplot[color=mycolor1,only marks,mark size=1pt,mark=o] plot table[row sep=crcr] {
-0.6392156862745098    -0.6392156862745098 \\
-0.6392156862745098    -0.6313725490196078 \\
-0.6392156862745098    -0.6235294117647059 \\
-0.6392156862745098    -0.6156862745098040 \\
-0.6392156862745098    -0.6078431372549020 \\
-0.6392156862745098    -0.6000000000000000 \\
-0.6392156862745098    -0.5921568627450980 \\
-0.6392156862745098    -0.5843137254901960 \\
-0.6392156862745098    -0.5764705882352941 \\
-0.6392156862745098    -0.5686274509803921 \\
-0.6392156862745098    -0.5607843137254902 \\
-0.6392156862745098    -0.5529411764705883 \\
-0.6313725490196078    -0.6392156862745098 \\
-0.6313725490196078    -0.6000000000000000 \\
-0.6313725490196078    -0.5921568627450980 \\
-0.6313725490196078    -0.5843137254901960 \\
-0.6313725490196078    -0.5764705882352941 \\
-0.6313725490196078    -0.5686274509803921 \\
-0.6313725490196078    -0.5607843137254902 \\
-0.6235294117647059    -0.6392156862745098 \\
-0.6235294117647059    -0.5764705882352941 \\
-0.6235294117647059    -0.5686274509803921 \\
-0.6156862745098040    -0.6392156862745098 \\
-0.6156862745098040    -0.5843137254901960 \\
-0.6156862745098040    -0.5764705882352941 \\
-0.6078431372549020    -0.6392156862745098 \\
-0.6078431372549020    -0.5921568627450980 \\
-0.6078431372549020    -0.5843137254901960 \\
-0.6000000000000000    -0.6392156862745098 \\
-0.6000000000000000    -0.6313725490196078 \\
-0.6000000000000000    -0.6000000000000000 \\
-0.6000000000000000    -0.5921568627450980 \\
-0.5921568627450980    -0.6392156862745098 \\
-0.5921568627450980    -0.6313725490196078 \\
-0.5921568627450980    -0.6078431372549020 \\
-0.5921568627450980    -0.6000000000000000 \\
-0.5843137254901960    -0.6392156862745098 \\
-0.5843137254901960    -0.6313725490196078 \\
-0.5843137254901960    -0.6156862745098040 \\
-0.5843137254901960    -0.6078431372549020 \\
-0.5764705882352941    -0.6392156862745098 \\
-0.5764705882352941    -0.6313725490196078 \\
-0.5764705882352941    -0.6235294117647059 \\
-0.5764705882352941    -0.6156862745098040 \\
-0.5686274509803921    -0.6392156862745098 \\
-0.5686274509803921    -0.6313725490196078 \\
-0.5686274509803921    -0.6235294117647059 \\
-0.5607843137254902    -0.6392156862745098 \\
-0.5607843137254902    -0.6313725490196078 \\
-0.5529411764705883    -0.6392156862745098 \\
 0.4901960784313726     0.6156862745098040 \\
 0.4901960784313726     0.6235294117647059 \\
 0.4980392156862745     0.6078431372549020 \\
 0.4980392156862745     0.6156862745098040 \\
 0.4980392156862745     0.6235294117647059 \\
 0.4980392156862745     0.6313725490196078 \\
 0.5058823529411764     0.5921568627450979 \\
 0.5058823529411764     0.6000000000000001 \\
 0.5058823529411764     0.6235294117647059 \\
 0.5058823529411764     0.6313725490196078 \\
 0.5137254901960784     0.5843137254901960 \\
 0.5137254901960784     0.5921568627450979 \\
 0.5137254901960784     0.6235294117647059 \\
 0.5215686274509803     0.5764705882352941 \\
 0.5215686274509803     0.6235294117647059 \\
 0.5294117647058822     0.5607843137254902 \\
 0.5294117647058822     0.5686274509803921 \\
 0.5294117647058822     0.6235294117647059 \\
 0.5372549019607844     0.5529411764705883 \\
 0.5372549019607844     0.5607843137254902 \\
 0.5372549019607844     0.6235294117647059 \\
 0.5450980392156863     0.5450980392156863 \\
 0.5450980392156863     0.5529411764705883 \\
 0.5450980392156863     0.6235294117647059 \\
 0.5529411764705883     0.5372549019607844 \\
 0.5529411764705883     0.5450980392156863 \\
 0.5529411764705883     0.6235294117647059 \\
 0.5607843137254902     0.5294117647058822 \\
 0.5607843137254902     0.5372549019607844 \\
 0.5607843137254902     0.6235294117647059 \\
 0.5686274509803921     0.5294117647058822 \\
 0.5686274509803921     0.6235294117647059 \\
 0.5764705882352941     0.5215686274509803 \\
 0.5764705882352941     0.6235294117647059 \\
 0.5764705882352941     0.6313725490196078 \\
 0.5843137254901960     0.5137254901960784 \\
 0.5843137254901960     0.6235294117647059 \\
 0.5843137254901960     0.6313725490196078 \\
 0.5921568627450979     0.5058823529411764 \\
 0.5921568627450979     0.5137254901960784 \\
 0.5921568627450979     0.6235294117647059 \\
 0.5921568627450979     0.6313725490196078 \\
 0.6000000000000001     0.5058823529411764 \\
 0.6000000000000001     0.6313725490196078 \\
 0.6078431372549020     0.4980392156862745 \\
 0.6078431372549020     0.6313725490196078 \\
 0.6156862745098040     0.4901960784313726 \\
 0.6156862745098040     0.4980392156862745 \\
 0.6156862745098040     0.6313725490196078 \\
 0.6235294117647059     0.4901960784313726 \\
 0.6235294117647059     0.4980392156862745 \\
 0.6235294117647059     0.5058823529411764 \\
 0.6235294117647059     0.5137254901960784 \\
 0.6235294117647059     0.5215686274509803 \\
 0.6235294117647059     0.5294117647058822 \\
 0.6235294117647059     0.5372549019607844 \\
 0.6235294117647059     0.5450980392156863 \\
 0.6235294117647059     0.5529411764705883 \\
 0.6235294117647059     0.5607843137254902 \\
 0.6235294117647059     0.5686274509803921 \\
 0.6235294117647059     0.5764705882352941 \\
 0.6235294117647059     0.5843137254901960 \\
 0.6235294117647059     0.5921568627450979 \\
 0.6235294117647059     0.6313725490196078 \\
 0.6313725490196078     0.4980392156862745 \\
 0.6313725490196078     0.5058823529411764 \\
 0.6313725490196078     0.5764705882352941 \\
 0.6313725490196078     0.5843137254901960 \\
 0.6313725490196078     0.5921568627450979 \\
 0.6313725490196078     0.6000000000000001 \\
 0.6313725490196078     0.6078431372549020 \\
 0.6313725490196078     0.6156862745098040 \\
 0.6313725490196078     0.6235294117647059 \\
 0.6313725490196078     0.6313725490196078 \\
};
\addplot[contour/draw color=mycolor2,contour prepared, contour prepared format=matlab, contour/labels=false] table[row sep=crcr] {%
0.5	769\\
-0.752941191196442	-0.749019622802734\\
-0.749019622802734	-0.752941191196442\\
-0.74117648601532	-0.752941191196442\\
-0.733333349227905	-0.752941191196442\\
-0.725490212440491	-0.752941191196442\\
-0.717647075653076	-0.752941191196442\\
-0.709803938865662	-0.752941191196442\\
-0.701960802078247	-0.752941191196442\\
-0.694117665290833	-0.752941191196442\\
-0.686274528503418	-0.752941191196442\\
-0.678431391716003	-0.752941191196442\\
-0.670588254928589	-0.752941191196442\\
-0.662745118141174	-0.752941191196442\\
-0.65490198135376	-0.752941191196442\\
-0.647058844566345	-0.752941191196442\\
-0.639215707778931	-0.752941191196442\\
-0.631372570991516	-0.752941191196442\\
-0.623529434204102	-0.752941191196442\\
-0.615686297416687	-0.752941191196442\\
-0.607843160629272	-0.752941191196442\\
-0.600000023841858	-0.752941191196442\\
-0.592156887054443	-0.752941191196442\\
-0.584313750267029	-0.752941191196442\\
-0.576470613479614	-0.752941191196442\\
-0.5686274766922	-0.752941191196442\\
-0.560784339904785	-0.752941191196442\\
-0.552941203117371	-0.752941191196442\\
-0.545098066329956	-0.752941191196442\\
-0.537254929542542	-0.752941191196442\\
-0.529411792755127	-0.752941191196442\\
-0.521568655967712	-0.752941191196442\\
-0.513725519180298	-0.752941191196442\\
-0.505882382392883	-0.752941191196442\\
-0.498039215803146	-0.752941191196442\\
-0.490196079015732	-0.752941191196442\\
-0.482352942228317	-0.752941191196442\\
-0.474509805440903	-0.752941191196442\\
-0.466666668653488	-0.752941191196442\\
-0.458823531866074	-0.752941191196442\\
-0.450980395078659	-0.752941191196442\\
-0.443137258291245	-0.752941191196442\\
-0.43529412150383	-0.752941191196442\\
-0.427450984716415	-0.752941191196442\\
-0.419607847929001	-0.752941191196442\\
-0.411764711141586	-0.752941191196442\\
-0.403921574354172	-0.752941191196442\\
-0.396078437566757	-0.752941191196442\\
-0.388235300779343	-0.752941191196442\\
-0.380392163991928	-0.752941191196442\\
-0.372549027204514	-0.752941191196442\\
-0.364705890417099	-0.752941191196442\\
-0.356862753629684	-0.752941191196442\\
-0.34901961684227	-0.752941191196442\\
-0.341176480054855	-0.752941191196442\\
-0.333333343267441	-0.752941191196442\\
-0.325490206480026	-0.752941191196442\\
-0.317647069692612	-0.752941191196442\\
-0.309803932905197	-0.752941191196442\\
-0.301960796117783	-0.752941191196442\\
-0.294117659330368	-0.752941191196442\\
-0.286274522542953	-0.752941191196442\\
-0.278431385755539	-0.752941191196442\\
-0.270588248968124	-0.752941191196442\\
-0.26274511218071	-0.752941191196442\\
-0.254901975393295	-0.752941191196442\\
-0.24705882370472	-0.752941191196442\\
-0.239215686917305	-0.752941191196442\\
-0.23137255012989	-0.752941191196442\\
-0.223529413342476	-0.752941191196442\\
-0.215686276555061	-0.752941191196442\\
-0.207843139767647	-0.752941191196442\\
-0.200000002980232	-0.752941191196442\\
-0.192156866192818	-0.752941191196442\\
-0.184313729405403	-0.752941191196442\\
-0.176470592617989	-0.752941191196442\\
-0.168627455830574	-0.752941191196442\\
-0.160784319043159	-0.752941191196442\\
-0.152941182255745	-0.752941191196442\\
-0.14509804546833	-0.752941191196442\\
-0.137254908680916	-0.752941191196442\\
-0.129411771893501	-0.752941191196442\\
-0.121568627655506	-0.752941191196442\\
-0.113725490868092	-0.752941191196442\\
-0.105882354080677	-0.752941191196442\\
-0.0980392172932625	-0.752941191196442\\
-0.0901960805058479	-0.752941191196442\\
-0.0823529437184334	-0.752941191196442\\
-0.0745098069310188	-0.752941191196442\\
-0.0666666701436043	-0.752941191196442\\
-0.0588235296308994	-0.752941191196442\\
-0.0509803928434849	-0.752941191196442\\
-0.0431372560560703	-0.752941191196442\\
-0.0352941192686558	-0.752941191196442\\
-0.0274509806185961	-0.752941191196442\\
-0.0196078438311815	-0.752941191196442\\
-0.0117647061124444	-0.752941191196442\\
-0.00392156885936856	-0.752941191196442\\
0.00392156885936856	-0.752941191196442\\
0.0117647061124444	-0.752941191196442\\
0.0196078438311815	-0.752941191196442\\
0.0274509806185961	-0.752941191196442\\
0.0352941192686558	-0.752941191196442\\
0.0431372560560703	-0.752941191196442\\
0.0509803928434849	-0.752941191196442\\
0.0588235296308994	-0.752941191196442\\
0.0666666701436043	-0.752941191196442\\
0.0745098069310188	-0.752941191196442\\
0.0823529437184334	-0.752941191196442\\
0.0901960805058479	-0.752941191196442\\
0.0980392172932625	-0.752941191196442\\
0.105882354080677	-0.752941191196442\\
0.113725490868092	-0.752941191196442\\
0.121568627655506	-0.752941191196442\\
0.129411771893501	-0.752941191196442\\
0.137254908680916	-0.752941191196442\\
0.14509804546833	-0.752941191196442\\
0.152941182255745	-0.752941191196442\\
0.160784319043159	-0.752941191196442\\
0.168627455830574	-0.752941191196442\\
0.176470592617989	-0.752941191196442\\
0.184313729405403	-0.752941191196442\\
0.192156866192818	-0.752941191196442\\
0.200000002980232	-0.752941191196442\\
0.207843139767647	-0.752941191196442\\
0.215686276555061	-0.752941191196442\\
0.223529413342476	-0.752941191196442\\
0.23137255012989	-0.752941191196442\\
0.239215686917305	-0.752941191196442\\
0.24705882370472	-0.752941191196442\\
0.254901975393295	-0.752941191196442\\
0.26274511218071	-0.752941191196442\\
0.270588248968124	-0.752941191196442\\
0.278431385755539	-0.752941191196442\\
0.286274522542953	-0.752941191196442\\
0.294117659330368	-0.752941191196442\\
0.301960796117783	-0.752941191196442\\
0.309803932905197	-0.752941191196442\\
0.317647069692612	-0.752941191196442\\
0.325490206480026	-0.752941191196442\\
0.333333343267441	-0.752941191196442\\
0.341176480054855	-0.752941191196442\\
0.34901961684227	-0.752941191196442\\
0.356862753629684	-0.752941191196442\\
0.364705890417099	-0.752941191196442\\
0.372549027204514	-0.752941191196442\\
0.380392163991928	-0.752941191196442\\
0.388235300779343	-0.752941191196442\\
0.396078437566757	-0.752941191196442\\
0.403921574354172	-0.752941191196442\\
0.411764711141586	-0.752941191196442\\
0.419607847929001	-0.752941191196442\\
0.427450984716415	-0.752941191196442\\
0.43529412150383	-0.752941191196442\\
0.443137258291245	-0.752941191196442\\
0.450980395078659	-0.752941191196442\\
0.458823531866074	-0.752941191196442\\
0.466666668653488	-0.752941191196442\\
0.474509805440903	-0.752941191196442\\
0.482352942228317	-0.752941191196442\\
0.490196079015732	-0.752941191196442\\
0.498039215803146	-0.752941191196442\\
0.505882382392883	-0.752941191196442\\
0.513725519180298	-0.752941191196442\\
0.521568655967712	-0.752941191196442\\
0.529411792755127	-0.752941191196442\\
0.537254929542542	-0.752941191196442\\
0.545098066329956	-0.752941191196442\\
0.552941203117371	-0.752941191196442\\
0.560784339904785	-0.752941191196442\\
0.5686274766922	-0.752941191196442\\
0.576470613479614	-0.752941191196442\\
0.584313750267029	-0.752941191196442\\
0.592156887054443	-0.752941191196442\\
0.600000023841858	-0.752941191196442\\
0.607843160629272	-0.752941191196442\\
0.615686297416687	-0.752941191196442\\
0.623529434204102	-0.752941191196442\\
0.631372570991516	-0.752941191196442\\
0.639215707778931	-0.752941191196442\\
0.647058844566345	-0.752941191196442\\
0.65490198135376	-0.752941191196442\\
0.662745118141174	-0.752941191196442\\
0.670588254928589	-0.752941191196442\\
0.678431391716003	-0.752941191196442\\
0.686274528503418	-0.752941191196442\\
0.694117665290833	-0.752941191196442\\
0.701960802078247	-0.752941191196442\\
0.709803938865662	-0.752941191196442\\
0.717647075653076	-0.752941191196442\\
0.725490212440491	-0.752941191196442\\
0.733333349227905	-0.752941191196442\\
0.74117648601532	-0.752941191196442\\
0.749019622802734	-0.752941191196442\\
0.752941191196442	-0.749019622802734\\
0.752941191196442	-0.74117648601532\\
0.752941191196442	-0.733333349227905\\
0.752941191196442	-0.725490212440491\\
0.752941191196442	-0.717647075653076\\
0.752941191196442	-0.709803938865662\\
0.752941191196442	-0.701960802078247\\
0.752941191196442	-0.694117665290833\\
0.752941191196442	-0.686274528503418\\
0.752941191196442	-0.678431391716003\\
0.752941191196442	-0.670588254928589\\
0.752941191196442	-0.662745118141174\\
0.752941191196442	-0.65490198135376\\
0.752941191196442	-0.647058844566345\\
0.752941191196442	-0.639215707778931\\
0.752941191196442	-0.631372570991516\\
0.752941191196442	-0.623529434204102\\
0.752941191196442	-0.615686297416687\\
0.752941191196442	-0.607843160629272\\
0.752941191196442	-0.600000023841858\\
0.752941191196442	-0.592156887054443\\
0.752941191196442	-0.584313750267029\\
0.752941191196442	-0.576470613479614\\
0.752941191196442	-0.5686274766922\\
0.752941191196442	-0.560784339904785\\
0.752941191196442	-0.552941203117371\\
0.752941191196442	-0.545098066329956\\
0.752941191196442	-0.537254929542542\\
0.752941191196442	-0.529411792755127\\
0.752941191196442	-0.521568655967712\\
0.752941191196442	-0.513725519180298\\
0.752941191196442	-0.505882382392883\\
0.752941191196442	-0.498039215803146\\
0.752941191196442	-0.490196079015732\\
0.752941191196442	-0.482352942228317\\
0.752941191196442	-0.474509805440903\\
0.752941191196442	-0.466666668653488\\
0.752941191196442	-0.458823531866074\\
0.752941191196442	-0.450980395078659\\
0.752941191196442	-0.443137258291245\\
0.752941191196442	-0.43529412150383\\
0.752941191196442	-0.427450984716415\\
0.752941191196442	-0.419607847929001\\
0.752941191196442	-0.411764711141586\\
0.752941191196442	-0.403921574354172\\
0.752941191196442	-0.396078437566757\\
0.752941191196442	-0.388235300779343\\
0.752941191196442	-0.380392163991928\\
0.752941191196442	-0.372549027204514\\
0.752941191196442	-0.364705890417099\\
0.752941191196442	-0.356862753629684\\
0.752941191196442	-0.34901961684227\\
0.752941191196442	-0.341176480054855\\
0.752941191196442	-0.333333343267441\\
0.752941191196442	-0.325490206480026\\
0.752941191196442	-0.317647069692612\\
0.752941191196442	-0.309803932905197\\
0.752941191196442	-0.301960796117783\\
0.752941191196442	-0.294117659330368\\
0.752941191196442	-0.286274522542953\\
0.752941191196442	-0.278431385755539\\
0.752941191196442	-0.270588248968124\\
0.752941191196442	-0.26274511218071\\
0.752941191196442	-0.254901975393295\\
0.752941191196442	-0.24705882370472\\
0.752941191196442	-0.239215686917305\\
0.752941191196442	-0.23137255012989\\
0.752941191196442	-0.223529413342476\\
0.752941191196442	-0.215686276555061\\
0.752941191196442	-0.207843139767647\\
0.752941191196442	-0.200000002980232\\
0.752941191196442	-0.192156866192818\\
0.752941191196442	-0.184313729405403\\
0.752941191196442	-0.176470592617989\\
0.752941191196442	-0.168627455830574\\
0.752941191196442	-0.160784319043159\\
0.752941191196442	-0.152941182255745\\
0.752941191196442	-0.14509804546833\\
0.752941191196442	-0.137254908680916\\
0.752941191196442	-0.129411771893501\\
0.752941191196442	-0.121568627655506\\
0.752941191196442	-0.113725490868092\\
0.752941191196442	-0.105882354080677\\
0.752941191196442	-0.0980392172932625\\
0.752941191196442	-0.0901960805058479\\
0.752941191196442	-0.0823529437184334\\
0.752941191196442	-0.0745098069310188\\
0.752941191196442	-0.0666666701436043\\
0.752941191196442	-0.0588235296308994\\
0.752941191196442	-0.0509803928434849\\
0.752941191196442	-0.0431372560560703\\
0.752941191196442	-0.0352941192686558\\
0.752941191196442	-0.0274509806185961\\
0.752941191196442	-0.0196078438311815\\
0.752941191196442	-0.0117647061124444\\
0.752941191196442	-0.00392156885936856\\
0.752941191196442	0.00392156885936856\\
0.752941191196442	0.0117647061124444\\
0.752941191196442	0.0196078438311815\\
0.752941191196442	0.0274509806185961\\
0.752941191196442	0.0352941192686558\\
0.752941191196442	0.0431372560560703\\
0.752941191196442	0.0509803928434849\\
0.752941191196442	0.0588235296308994\\
0.752941191196442	0.0666666701436043\\
0.752941191196442	0.0745098069310188\\
0.752941191196442	0.0823529437184334\\
0.752941191196442	0.0901960805058479\\
0.752941191196442	0.0980392172932625\\
0.752941191196442	0.105882354080677\\
0.752941191196442	0.113725490868092\\
0.752941191196442	0.121568627655506\\
0.752941191196442	0.129411771893501\\
0.752941191196442	0.137254908680916\\
0.752941191196442	0.14509804546833\\
0.752941191196442	0.152941182255745\\
0.752941191196442	0.160784319043159\\
0.752941191196442	0.168627455830574\\
0.752941191196442	0.176470592617989\\
0.752941191196442	0.184313729405403\\
0.752941191196442	0.192156866192818\\
0.752941191196442	0.200000002980232\\
0.752941191196442	0.207843139767647\\
0.752941191196442	0.215686276555061\\
0.752941191196442	0.223529413342476\\
0.752941191196442	0.23137255012989\\
0.752941191196442	0.239215686917305\\
0.752941191196442	0.24705882370472\\
0.752941191196442	0.254901975393295\\
0.752941191196442	0.26274511218071\\
0.752941191196442	0.270588248968124\\
0.752941191196442	0.278431385755539\\
0.752941191196442	0.286274522542953\\
0.752941191196442	0.294117659330368\\
0.752941191196442	0.301960796117783\\
0.752941191196442	0.309803932905197\\
0.752941191196442	0.317647069692612\\
0.752941191196442	0.325490206480026\\
0.752941191196442	0.333333343267441\\
0.752941191196442	0.341176480054855\\
0.752941191196442	0.34901961684227\\
0.752941191196442	0.356862753629684\\
0.752941191196442	0.364705890417099\\
0.752941191196442	0.372549027204514\\
0.752941191196442	0.380392163991928\\
0.752941191196442	0.388235300779343\\
0.752941191196442	0.396078437566757\\
0.752941191196442	0.403921574354172\\
0.752941191196442	0.411764711141586\\
0.752941191196442	0.419607847929001\\
0.752941191196442	0.427450984716415\\
0.752941191196442	0.43529412150383\\
0.752941191196442	0.443137258291245\\
0.752941191196442	0.450980395078659\\
0.752941191196442	0.458823531866074\\
0.752941191196442	0.466666668653488\\
0.752941191196442	0.474509805440903\\
0.752941191196442	0.482352942228317\\
0.752941191196442	0.490196079015732\\
0.752941191196442	0.498039215803146\\
0.752941191196442	0.505882382392883\\
0.752941191196442	0.513725519180298\\
0.752941191196442	0.521568655967712\\
0.752941191196442	0.529411792755127\\
0.752941191196442	0.537254929542542\\
0.752941191196442	0.545098066329956\\
0.752941191196442	0.552941203117371\\
0.752941191196442	0.560784339904785\\
0.752941191196442	0.5686274766922\\
0.752941191196442	0.576470613479614\\
0.752941191196442	0.584313750267029\\
0.752941191196442	0.592156887054443\\
0.752941191196442	0.600000023841858\\
0.752941191196442	0.607843160629272\\
0.752941191196442	0.615686297416687\\
0.752941191196442	0.623529434204102\\
0.752941191196442	0.631372570991516\\
0.752941191196442	0.639215707778931\\
0.752941191196442	0.647058844566345\\
0.752941191196442	0.65490198135376\\
0.752941191196442	0.662745118141174\\
0.752941191196442	0.670588254928589\\
0.752941191196442	0.678431391716003\\
0.752941191196442	0.686274528503418\\
0.752941191196442	0.694117665290833\\
0.752941191196442	0.701960802078247\\
0.752941191196442	0.709803938865662\\
0.752941191196442	0.717647075653076\\
0.752941191196442	0.725490212440491\\
0.752941191196442	0.733333349227905\\
0.752941191196442	0.74117648601532\\
0.752941191196442	0.749019622802734\\
0.749019622802734	0.752941191196442\\
0.74117648601532	0.752941191196442\\
0.733333349227905	0.752941191196442\\
0.725490212440491	0.752941191196442\\
0.717647075653076	0.752941191196442\\
0.709803938865662	0.752941191196442\\
0.701960802078247	0.752941191196442\\
0.694117665290833	0.752941191196442\\
0.686274528503418	0.752941191196442\\
0.678431391716003	0.752941191196442\\
0.670588254928589	0.752941191196442\\
0.662745118141174	0.752941191196442\\
0.65490198135376	0.752941191196442\\
0.647058844566345	0.752941191196442\\
0.639215707778931	0.752941191196442\\
0.631372570991516	0.752941191196442\\
0.623529434204102	0.752941191196442\\
0.615686297416687	0.752941191196442\\
0.607843160629272	0.752941191196442\\
0.600000023841858	0.752941191196442\\
0.592156887054443	0.752941191196442\\
0.584313750267029	0.752941191196442\\
0.576470613479614	0.752941191196442\\
0.5686274766922	0.752941191196442\\
0.560784339904785	0.752941191196442\\
0.552941203117371	0.752941191196442\\
0.545098066329956	0.752941191196442\\
0.537254929542542	0.752941191196442\\
0.529411792755127	0.752941191196442\\
0.521568655967712	0.752941191196442\\
0.513725519180298	0.752941191196442\\
0.505882382392883	0.752941191196442\\
0.498039215803146	0.752941191196442\\
0.490196079015732	0.752941191196442\\
0.482352942228317	0.752941191196442\\
0.474509805440903	0.752941191196442\\
0.466666668653488	0.752941191196442\\
0.458823531866074	0.752941191196442\\
0.450980395078659	0.752941191196442\\
0.443137258291245	0.752941191196442\\
0.43529412150383	0.752941191196442\\
0.427450984716415	0.752941191196442\\
0.419607847929001	0.752941191196442\\
0.411764711141586	0.752941191196442\\
0.403921574354172	0.752941191196442\\
0.396078437566757	0.752941191196442\\
0.388235300779343	0.752941191196442\\
0.380392163991928	0.752941191196442\\
0.372549027204514	0.752941191196442\\
0.364705890417099	0.752941191196442\\
0.356862753629684	0.752941191196442\\
0.34901961684227	0.752941191196442\\
0.341176480054855	0.752941191196442\\
0.333333343267441	0.752941191196442\\
0.325490206480026	0.752941191196442\\
0.317647069692612	0.752941191196442\\
0.309803932905197	0.752941191196442\\
0.301960796117783	0.752941191196442\\
0.294117659330368	0.752941191196442\\
0.286274522542953	0.752941191196442\\
0.278431385755539	0.752941191196442\\
0.270588248968124	0.752941191196442\\
0.26274511218071	0.752941191196442\\
0.254901975393295	0.752941191196442\\
0.24705882370472	0.752941191196442\\
0.239215686917305	0.752941191196442\\
0.23137255012989	0.752941191196442\\
0.223529413342476	0.752941191196442\\
0.215686276555061	0.752941191196442\\
0.207843139767647	0.752941191196442\\
0.200000002980232	0.752941191196442\\
0.192156866192818	0.752941191196442\\
0.184313729405403	0.752941191196442\\
0.176470592617989	0.752941191196442\\
0.168627455830574	0.752941191196442\\
0.160784319043159	0.752941191196442\\
0.152941182255745	0.752941191196442\\
0.14509804546833	0.752941191196442\\
0.137254908680916	0.752941191196442\\
0.129411771893501	0.752941191196442\\
0.121568627655506	0.752941191196442\\
0.113725490868092	0.752941191196442\\
0.105882354080677	0.752941191196442\\
0.0980392172932625	0.752941191196442\\
0.0901960805058479	0.752941191196442\\
0.0823529437184334	0.752941191196442\\
0.0745098069310188	0.752941191196442\\
0.0666666701436043	0.752941191196442\\
0.0588235296308994	0.752941191196442\\
0.0509803928434849	0.752941191196442\\
0.0431372560560703	0.752941191196442\\
0.0352941192686558	0.752941191196442\\
0.0274509806185961	0.752941191196442\\
0.0196078438311815	0.752941191196442\\
0.0117647061124444	0.752941191196442\\
0.00392156885936856	0.752941191196442\\
-0.00392156885936856	0.752941191196442\\
-0.0117647061124444	0.752941191196442\\
-0.0196078438311815	0.752941191196442\\
-0.0274509806185961	0.752941191196442\\
-0.0352941192686558	0.752941191196442\\
-0.0431372560560703	0.752941191196442\\
-0.0509803928434849	0.752941191196442\\
-0.0588235296308994	0.752941191196442\\
-0.0666666701436043	0.752941191196442\\
-0.0745098069310188	0.752941191196442\\
-0.0823529437184334	0.752941191196442\\
-0.0901960805058479	0.752941191196442\\
-0.0980392172932625	0.752941191196442\\
-0.105882354080677	0.752941191196442\\
-0.113725490868092	0.752941191196442\\
-0.121568627655506	0.752941191196442\\
-0.129411771893501	0.752941191196442\\
-0.137254908680916	0.752941191196442\\
-0.14509804546833	0.752941191196442\\
-0.152941182255745	0.752941191196442\\
-0.160784319043159	0.752941191196442\\
-0.168627455830574	0.752941191196442\\
-0.176470592617989	0.752941191196442\\
-0.184313729405403	0.752941191196442\\
-0.192156866192818	0.752941191196442\\
-0.200000002980232	0.752941191196442\\
-0.207843139767647	0.752941191196442\\
-0.215686276555061	0.752941191196442\\
-0.223529413342476	0.752941191196442\\
-0.23137255012989	0.752941191196442\\
-0.239215686917305	0.752941191196442\\
-0.24705882370472	0.752941191196442\\
-0.254901975393295	0.752941191196442\\
-0.26274511218071	0.752941191196442\\
-0.270588248968124	0.752941191196442\\
-0.278431385755539	0.752941191196442\\
-0.286274522542953	0.752941191196442\\
-0.294117659330368	0.752941191196442\\
-0.301960796117783	0.752941191196442\\
-0.309803932905197	0.752941191196442\\
-0.317647069692612	0.752941191196442\\
-0.325490206480026	0.752941191196442\\
-0.333333343267441	0.752941191196442\\
-0.341176480054855	0.752941191196442\\
-0.34901961684227	0.752941191196442\\
-0.356862753629684	0.752941191196442\\
-0.364705890417099	0.752941191196442\\
-0.372549027204514	0.752941191196442\\
-0.380392163991928	0.752941191196442\\
-0.388235300779343	0.752941191196442\\
-0.396078437566757	0.752941191196442\\
-0.403921574354172	0.752941191196442\\
-0.411764711141586	0.752941191196442\\
-0.419607847929001	0.752941191196442\\
-0.427450984716415	0.752941191196442\\
-0.43529412150383	0.752941191196442\\
-0.443137258291245	0.752941191196442\\
-0.450980395078659	0.752941191196442\\
-0.458823531866074	0.752941191196442\\
-0.466666668653488	0.752941191196442\\
-0.474509805440903	0.752941191196442\\
-0.482352942228317	0.752941191196442\\
-0.490196079015732	0.752941191196442\\
-0.498039215803146	0.752941191196442\\
-0.505882382392883	0.752941191196442\\
-0.513725519180298	0.752941191196442\\
-0.521568655967712	0.752941191196442\\
-0.529411792755127	0.752941191196442\\
-0.537254929542542	0.752941191196442\\
-0.545098066329956	0.752941191196442\\
-0.552941203117371	0.752941191196442\\
-0.560784339904785	0.752941191196442\\
-0.5686274766922	0.752941191196442\\
-0.576470613479614	0.752941191196442\\
-0.584313750267029	0.752941191196442\\
-0.592156887054443	0.752941191196442\\
-0.600000023841858	0.752941191196442\\
-0.607843160629272	0.752941191196442\\
-0.615686297416687	0.752941191196442\\
-0.623529434204102	0.752941191196442\\
-0.631372570991516	0.752941191196442\\
-0.639215707778931	0.752941191196442\\
-0.647058844566345	0.752941191196442\\
-0.65490198135376	0.752941191196442\\
-0.662745118141174	0.752941191196442\\
-0.670588254928589	0.752941191196442\\
-0.678431391716003	0.752941191196442\\
-0.686274528503418	0.752941191196442\\
-0.694117665290833	0.752941191196442\\
-0.701960802078247	0.752941191196442\\
-0.709803938865662	0.752941191196442\\
-0.717647075653076	0.752941191196442\\
-0.725490212440491	0.752941191196442\\
-0.733333349227905	0.752941191196442\\
-0.74117648601532	0.752941191196442\\
-0.749019622802734	0.752941191196442\\
-0.752941191196442	0.749019622802734\\
-0.752941191196442	0.74117648601532\\
-0.752941191196442	0.733333349227905\\
-0.752941191196442	0.725490212440491\\
-0.752941191196442	0.717647075653076\\
-0.752941191196442	0.709803938865662\\
-0.752941191196442	0.701960802078247\\
-0.752941191196442	0.694117665290833\\
-0.752941191196442	0.686274528503418\\
-0.752941191196442	0.678431391716003\\
-0.752941191196442	0.670588254928589\\
-0.752941191196442	0.662745118141174\\
-0.752941191196442	0.65490198135376\\
-0.752941191196442	0.647058844566345\\
-0.752941191196442	0.639215707778931\\
-0.752941191196442	0.631372570991516\\
-0.752941191196442	0.623529434204102\\
-0.752941191196442	0.615686297416687\\
-0.752941191196442	0.607843160629272\\
-0.752941191196442	0.600000023841858\\
-0.752941191196442	0.592156887054443\\
-0.752941191196442	0.584313750267029\\
-0.752941191196442	0.576470613479614\\
-0.752941191196442	0.5686274766922\\
-0.752941191196442	0.560784339904785\\
-0.752941191196442	0.552941203117371\\
-0.752941191196442	0.545098066329956\\
-0.752941191196442	0.537254929542542\\
-0.752941191196442	0.529411792755127\\
-0.752941191196442	0.521568655967712\\
-0.752941191196442	0.513725519180298\\
-0.752941191196442	0.505882382392883\\
-0.752941191196442	0.498039215803146\\
-0.752941191196442	0.490196079015732\\
-0.752941191196442	0.482352942228317\\
-0.752941191196442	0.474509805440903\\
-0.752941191196442	0.466666668653488\\
-0.752941191196442	0.458823531866074\\
-0.752941191196442	0.450980395078659\\
-0.752941191196442	0.443137258291245\\
-0.752941191196442	0.43529412150383\\
-0.752941191196442	0.427450984716415\\
-0.752941191196442	0.419607847929001\\
-0.752941191196442	0.411764711141586\\
-0.752941191196442	0.403921574354172\\
-0.752941191196442	0.396078437566757\\
-0.752941191196442	0.388235300779343\\
-0.752941191196442	0.380392163991928\\
-0.752941191196442	0.372549027204514\\
-0.752941191196442	0.364705890417099\\
-0.752941191196442	0.356862753629684\\
-0.752941191196442	0.34901961684227\\
-0.752941191196442	0.341176480054855\\
-0.752941191196442	0.333333343267441\\
-0.752941191196442	0.325490206480026\\
-0.752941191196442	0.317647069692612\\
-0.752941191196442	0.309803932905197\\
-0.752941191196442	0.301960796117783\\
-0.752941191196442	0.294117659330368\\
-0.752941191196442	0.286274522542953\\
-0.752941191196442	0.278431385755539\\
-0.752941191196442	0.270588248968124\\
-0.752941191196442	0.26274511218071\\
-0.752941191196442	0.254901975393295\\
-0.752941191196442	0.24705882370472\\
-0.752941191196442	0.239215686917305\\
-0.752941191196442	0.23137255012989\\
-0.752941191196442	0.223529413342476\\
-0.752941191196442	0.215686276555061\\
-0.752941191196442	0.207843139767647\\
-0.752941191196442	0.200000002980232\\
-0.752941191196442	0.192156866192818\\
-0.752941191196442	0.184313729405403\\
-0.752941191196442	0.176470592617989\\
-0.752941191196442	0.168627455830574\\
-0.752941191196442	0.160784319043159\\
-0.752941191196442	0.152941182255745\\
-0.752941191196442	0.14509804546833\\
-0.752941191196442	0.137254908680916\\
-0.752941191196442	0.129411771893501\\
-0.752941191196442	0.121568627655506\\
-0.752941191196442	0.113725490868092\\
-0.752941191196442	0.105882354080677\\
-0.752941191196442	0.0980392172932625\\
-0.752941191196442	0.0901960805058479\\
-0.752941191196442	0.0823529437184334\\
-0.752941191196442	0.0745098069310188\\
-0.752941191196442	0.0666666701436043\\
-0.752941191196442	0.0588235296308994\\
-0.752941191196442	0.0509803928434849\\
-0.752941191196442	0.0431372560560703\\
-0.752941191196442	0.0352941192686558\\
-0.752941191196442	0.0274509806185961\\
-0.752941191196442	0.0196078438311815\\
-0.752941191196442	0.0117647061124444\\
-0.752941191196442	0.00392156885936856\\
-0.752941191196442	-0.00392156885936856\\
-0.752941191196442	-0.0117647061124444\\
-0.752941191196442	-0.0196078438311815\\
-0.752941191196442	-0.0274509806185961\\
-0.752941191196442	-0.0352941192686558\\
-0.752941191196442	-0.0431372560560703\\
-0.752941191196442	-0.0509803928434849\\
-0.752941191196442	-0.0588235296308994\\
-0.752941191196442	-0.0666666701436043\\
-0.752941191196442	-0.0745098069310188\\
-0.752941191196442	-0.0823529437184334\\
-0.752941191196442	-0.0901960805058479\\
-0.752941191196442	-0.0980392172932625\\
-0.752941191196442	-0.105882354080677\\
-0.752941191196442	-0.113725490868092\\
-0.752941191196442	-0.121568627655506\\
-0.752941191196442	-0.129411771893501\\
-0.752941191196442	-0.137254908680916\\
-0.752941191196442	-0.14509804546833\\
-0.752941191196442	-0.152941182255745\\
-0.752941191196442	-0.160784319043159\\
-0.752941191196442	-0.168627455830574\\
-0.752941191196442	-0.176470592617989\\
-0.752941191196442	-0.184313729405403\\
-0.752941191196442	-0.192156866192818\\
-0.752941191196442	-0.200000002980232\\
-0.752941191196442	-0.207843139767647\\
-0.752941191196442	-0.215686276555061\\
-0.752941191196442	-0.223529413342476\\
-0.752941191196442	-0.23137255012989\\
-0.752941191196442	-0.239215686917305\\
-0.752941191196442	-0.24705882370472\\
-0.752941191196442	-0.254901975393295\\
-0.752941191196442	-0.26274511218071\\
-0.752941191196442	-0.270588248968124\\
-0.752941191196442	-0.278431385755539\\
-0.752941191196442	-0.286274522542953\\
-0.752941191196442	-0.294117659330368\\
-0.752941191196442	-0.301960796117783\\
-0.752941191196442	-0.309803932905197\\
-0.752941191196442	-0.317647069692612\\
-0.752941191196442	-0.325490206480026\\
-0.752941191196442	-0.333333343267441\\
-0.752941191196442	-0.341176480054855\\
-0.752941191196442	-0.34901961684227\\
-0.752941191196442	-0.356862753629684\\
-0.752941191196442	-0.364705890417099\\
-0.752941191196442	-0.372549027204514\\
-0.752941191196442	-0.380392163991928\\
-0.752941191196442	-0.388235300779343\\
-0.752941191196442	-0.396078437566757\\
-0.752941191196442	-0.403921574354172\\
-0.752941191196442	-0.411764711141586\\
-0.752941191196442	-0.419607847929001\\
-0.752941191196442	-0.427450984716415\\
-0.752941191196442	-0.43529412150383\\
-0.752941191196442	-0.443137258291245\\
-0.752941191196442	-0.450980395078659\\
-0.752941191196442	-0.458823531866074\\
-0.752941191196442	-0.466666668653488\\
-0.752941191196442	-0.474509805440903\\
-0.752941191196442	-0.482352942228317\\
-0.752941191196442	-0.490196079015732\\
-0.752941191196442	-0.498039215803146\\
-0.752941191196442	-0.505882382392883\\
-0.752941191196442	-0.513725519180298\\
-0.752941191196442	-0.521568655967712\\
-0.752941191196442	-0.529411792755127\\
-0.752941191196442	-0.537254929542542\\
-0.752941191196442	-0.545098066329956\\
-0.752941191196442	-0.552941203117371\\
-0.752941191196442	-0.560784339904785\\
-0.752941191196442	-0.5686274766922\\
-0.752941191196442	-0.576470613479614\\
-0.752941191196442	-0.584313750267029\\
-0.752941191196442	-0.592156887054443\\
-0.752941191196442	-0.600000023841858\\
-0.752941191196442	-0.607843160629272\\
-0.752941191196442	-0.615686297416687\\
-0.752941191196442	-0.623529434204102\\
-0.752941191196442	-0.631372570991516\\
-0.752941191196442	-0.639215707778931\\
-0.752941191196442	-0.647058844566345\\
-0.752941191196442	-0.65490198135376\\
-0.752941191196442	-0.662745118141174\\
-0.752941191196442	-0.670588254928589\\
-0.752941191196442	-0.678431391716003\\
-0.752941191196442	-0.686274528503418\\
-0.752941191196442	-0.694117665290833\\
-0.752941191196442	-0.701960802078247\\
-0.752941191196442	-0.709803938865662\\
-0.752941191196442	-0.717647075653076\\
-0.752941191196442	-0.725490212440491\\
-0.752941191196442	-0.733333349227905\\
-0.752941191196442	-0.74117648601532\\
-0.752941191196442	-0.749019622802734\\
};
\end{axis}
\end{tikzpicture}%

%% file: neumann_control.tikz
\definecolor{mycolor1}{rgb}{0.00000,0.44700,0.74100}%
\begin{tikzpicture}

\begin{axis}[%
width=\textwidth,
height=0.3\textheight,
xmin=1,
xmax=1024,
ymin=0,
ymax=0.002,
ytick={0,.002},
major tick length=0,
xtick={129,385,641,897},xticklabels={$x_1=-1$,$x_2=-1$,$x_1=1$,$x_2=1$},
extra x ticks={265.5,512.5,768.5},extra x tick style={grid=major},extra x tick labels=\empty,
axis on top]
\addplot [color=mycolor1,solid,forget plot]
  table[row sep=crcr]{%
1	0\\
2	0\\
3	0\\
4	0\\
5	0\\
6	0\\
7	0\\
8	0\\
9	0\\
10	0\\
11	0\\
12	0\\
13	0\\
14	0\\
15	0\\
16	0\\
17	0\\
18	0\\
19	0\\
20	0\\
21	0\\
22	0\\
23	0\\
24	0\\
25	0\\
26	0\\
27	0\\
28	0\\
29	0\\
30	0\\
31	0\\
32	0\\
33	0\\
34	0\\
35	0\\
36	0\\
37	0\\
38	0\\
39	0\\
40	0\\
41	0\\
42	0\\
43	0\\
44	0\\
45	0\\
46	0\\
47	0\\
48	0\\
49	0\\
50	0\\
51	0\\
52	0\\
53	0\\
54	0\\
55	0\\
56	0\\
57	0\\
58	0\\
59	0\\
60	0\\
61	0\\
62	0\\
63	0\\
64	0\\
65	0\\
66	0\\
67	0\\
68	0\\
69	0\\
70	0\\
71	0\\
72	0\\
73	0\\
74	0\\
75	0\\
76	0\\
77	0\\
78	0\\
79	0\\
80	0\\
81	0\\
82	0\\
83	0\\
84	0\\
85	0\\
86	0\\
87	0\\
88	0\\
89	0\\
90	0\\
91	0\\
92	0\\
93	0\\
94	0\\
95	0\\
96	0\\
97	0\\
98	0\\
99	0\\
100	0\\
101	0\\
102	0\\
103	0\\
104	0\\
105	0\\
106	0\\
107	0\\
108	0\\
109	0\\
110	0\\
111	0\\
112	0\\
113	0\\
114	0\\
115	0\\
116	0\\
117	0\\
118	0\\
119	0\\
120	0\\
121	0\\
122	0\\
123	0\\
124	0\\
125	0\\
126	0\\
127	0\\
128	0\\
129	0\\
130	0\\
131	0\\
132	0\\
133	0\\
134	0\\
135	0\\
136	0\\
137	0\\
138	0\\
139	0\\
140	0\\
141	0\\
142	0\\
143	0\\
144	0\\
145	0\\
146	0\\
147	0\\
148	0\\
149	0\\
150	0\\
151	0\\
152	0\\
153	0\\
154	0\\
155	0\\
156	0\\
157	0\\
158	0\\
159	0\\
160	0\\
161	0\\
162	0\\
163	0\\
164	0\\
165	0\\
166	0\\
167	0\\
168	0\\
169	0\\
170	0\\
171	0\\
172	0\\
173	0\\
174	0\\
175	0\\
176	0\\
177	0\\
178	0\\
179	0\\
180	0\\
181	0\\
182	0\\
183	0\\
184	0\\
185	0\\
186	0\\
187	0\\
188	0\\
189	0\\
190	0\\
191	0\\
192	0\\
193	0\\
194	0\\
195	0\\
196	0\\
197	0\\
198	0\\
199	0\\
200	0\\
201	0\\
202	0\\
203	0\\
204	0\\
205	0\\
206	0\\
207	0\\
208	0\\
209	0\\
210	0\\
211	0\\
212	0\\
213	0\\
214	0\\
215	0\\
216	0\\
217	0\\
218	0\\
219	0\\
220	0\\
221	0\\
222	0\\
223	0\\
224	0\\
225	0\\
226	0\\
227	0\\
228	0\\
229	0\\
230	0\\
231	0\\
232	0\\
233	0\\
234	0\\
235	0\\
236	0\\
237	0\\
238	0\\
239	0\\
240	0\\
241	0\\
242	0\\
243	0\\
244	0\\
245	0\\
246	0\\
247	0\\
248	0\\
249	0\\
250	0\\
251	0\\
252	0\\
253	0\\
254	0\\
255	0\\
256	0\\
257	0\\
258	0\\
259	0\\
260	0\\
261	0\\
262	0\\
263	0\\
264	0\\
265	0\\
266	0\\
267	0\\
268	0\\
269	0\\
270	0\\
271	0\\
272	0\\
273	0\\
274	0\\
275	0\\
276	0\\
277	0\\
278	0\\
279	0\\
280	0\\
281	0\\
282	0\\
283	0\\
284	0\\
285	0\\
286	0\\
287	0\\
288	0\\
289	0\\
290	0\\
291	0\\
292	0\\
293	0\\
294	0\\
295	0\\
296	0\\
297	0\\
298	0\\
299	0\\
300	0\\
301	0\\
302	0\\
303	0\\
304	0\\
305	0\\
306	0\\
307	0\\
308	0\\
309	0\\
310	0\\
311	0\\
312	0\\
313	0\\
314	0\\
315	0\\
316	0\\
317	0\\
318	0\\
319	0\\
320	0\\
321	0\\
322	0\\
323	0\\
324	0\\
325	0\\
326	0\\
327	0\\
328	0\\
329	0\\
330	0\\
331	0\\
332	0\\
333	0\\
334	0\\
335	0\\
336	0\\
337	0\\
338	0\\
339	0\\
340	0\\
341	0\\
342	0\\
343	0\\
344	0\\
345	0\\
346	0\\
347	0\\
348	0\\
349	0\\
350	0\\
351	0\\
352	0\\
353	0\\
354	0\\
355	0\\
356	0\\
357	0\\
358	0\\
359	0\\
360	0\\
361	0\\
362	0\\
363	0\\
364	0\\
365	0\\
366	0\\
367	0\\
368	0\\
369	0\\
370	0\\
371	0\\
372	0\\
373	0\\
374	0\\
375	0\\
376	0\\
377	0\\
378	0\\
379	0\\
380	0\\
381	0\\
382	0\\
383	0\\
384	0\\
385	0\\
386	0\\
387	0\\
388	0\\
389	0\\
390	0\\
391	0\\
392	0\\
393	0\\
394	0\\
395	0\\
396	0\\
397	0\\
398	0\\
399	0\\
400	0\\
401	0\\
402	0\\
403	0\\
404	0\\
405	0\\
406	0\\
407	0\\
408	0\\
409	0\\
410	0\\
411	0\\
412	0\\
413	0\\
414	0\\
415	0\\
416	0\\
417	0\\
418	0\\
419	0\\
420	0\\
421	0\\
422	0\\
423	0\\
424	0\\
425	0\\
426	0\\
427	0\\
428	0\\
429	0\\
430	0\\
431	0\\
432	0\\
433	0\\
434	0\\
435	0\\
436	0\\
437	0\\
438	0\\
439	0\\
440	0\\
441	0\\
442	0\\
443	0\\
444	0\\
445	0\\
446	0\\
447	0\\
448	0\\
449	0\\
450	0\\
451	0\\
452	0\\
453	0\\
454	0\\
455	0\\
456	0\\
457	0\\
458	0\\
459	0\\
460	0\\
461	0\\
462	0\\
463	0\\
464	0\\
465	0\\
466	0\\
467	0\\
468	0\\
469	0\\
470	0\\
471	0\\
472	0\\
473	0\\
474	0\\
475	0\\
476	0\\
477	0\\
478	0\\
479	0\\
480	0\\
481	0\\
482	0\\
483	0\\
484	0\\
485	0\\
486	0\\
487	0\\
488	0\\
489	0\\
490	0\\
491	0\\
492	0\\
493	0\\
494	0\\
495	0\\
496	0\\
497	0\\
498	0\\
499	0\\
500	0\\
501	0\\
502	0\\
503	0\\
504	0\\
505	0\\
506	0\\
507	0\\
508	0\\
509	0\\
510	0\\
511	0\\
512	0\\
513	0\\
514	0\\
515	0\\
516	0\\
517	0\\
518	0\\
519	0\\
520	0\\
521	0\\
522	0\\
523	0\\
524	0\\
525	0\\
526	0\\
527	0\\
528	0\\
529	0\\
530	0\\
531	0\\
532	0\\
533	0\\
534	0\\
535	0\\
536	0\\
537	0\\
538	0\\
539	0\\
540	0\\
541	0\\
542	0\\
543	0\\
544	0\\
545	0\\
546	0\\
547	0\\
548	0\\
549	0\\
550	0\\
551	0\\
552	0\\
553	0\\
554	0\\
555	0\\
556	0\\
557	0\\
558	0\\
559	0\\
560	0\\
561	0\\
562	0\\
563	0\\
564	0\\
565	0\\
566	0\\
567	0\\
568	0\\
569	0\\
570	0\\
571	0\\
572	0\\
573	0\\
574	0\\
575	0\\
576	0\\
577	0\\
578	0\\
579	0\\
580	0\\
581	0\\
582	0\\
583	0\\
584	0\\
585	0\\
586	0\\
587	0\\
588	0\\
589	0\\
590	0\\
591	0\\
592	0\\
593	0\\
594	0\\
595	0\\
596	0\\
597	0\\
598	0\\
599	0\\
600	0\\
601	0\\
602	0\\
603	0\\
604	0\\
605	0\\
606	0\\
607	0\\
608	0\\
609	0\\
610	0\\
611	0\\
612	0\\
613	0\\
614	0\\
615	0\\
616	0\\
617	0\\
618	0\\
619	0\\
620	0\\
621	0\\
622	0\\
623	0\\
624	0\\
625	0\\
626	0\\
627	0\\
628	0\\
629	0\\
630	0\\
631	0\\
632	0\\
633	0\\
634	0\\
635	0\\
636	0\\
637	0\\
638	0\\
639	0\\
640	0\\
641	0\\
642	0\\
643	0\\
644	0\\
645	0\\
646	0\\
647	0\\
648	0\\
649	0\\
650	0\\
651	0\\
652	0\\
653	0\\
654	0\\
655	0\\
656	0\\
657	0\\
658	0\\
659	0\\
660	0\\
661	0\\
662	0\\
663	0\\
664	0\\
665	0\\
666	0\\
667	0\\
668	0\\
669	0\\
670	0\\
671	0\\
672	0\\
673	0\\
674	0\\
675	0\\
676	0\\
677	0\\
678	0\\
679	0\\
680	0\\
681	0\\
682	0\\
683	0\\
684	0\\
685	0\\
686	0\\
687	0\\
688	0\\
689	0\\
690	0\\
691	0\\
692	0\\
693	0\\
694	0\\
695	0\\
696	0\\
697	0\\
698	0\\
699	0\\
700	0\\
701	0\\
702	0\\
703	0\\
704	0\\
705	0\\
706	0\\
707	0\\
708	0\\
709	0\\
710	0\\
711	0\\
712	0\\
713	0\\
714	0\\
715	0\\
716	0\\
717	0\\
718	0\\
719	0\\
720	0\\
721	0\\
722	0\\
723	0\\
724	0\\
725	0\\
726	0\\
727	0\\
728	0\\
729	0\\
730	0\\
731	0\\
732	0\\
733	0\\
734	0\\
735	0.00190041755229524\\
736	0\\
737	0\\
738	0\\
739	0\\
740	0\\
741	0\\
742	0\\
743	0\\
744	0\\
745	0\\
746	0\\
747	0\\
748	0\\
749	0\\
750	0\\
751	0\\
752	0\\
753	0\\
754	0\\
755	0\\
756	0\\
757	0\\
758	0\\
759	0\\
760	0\\
761	0\\
762	0\\
763	0\\
764	0\\
765	0\\
766	0\\
767	0\\
768	0\\
769	0\\
770	0\\
771	0\\
772	0\\
773	0\\
774	0\\
775	0\\
776	0\\
777	0\\
778	0\\
779	0\\
780	0\\
781	0\\
782	0\\
783	0\\
784	0\\
785	0\\
786	0\\
787	0\\
788	0\\
789	0\\
790	0\\
791	0\\
792	0\\
793	0\\
794	0\\
795	0\\
796	0\\
797	0\\
798	0\\
799	0\\
800	0\\
801	0\\
802	0.00190041755229684\\
803	0\\
804	0\\
805	0\\
806	0\\
807	0\\
808	0\\
809	0\\
810	0\\
811	0\\
812	0\\
813	0\\
814	0\\
815	0\\
816	0\\
817	0\\
818	0\\
819	0\\
820	0\\
821	0\\
822	0\\
823	0\\
824	0\\
825	0\\
826	0\\
827	0\\
828	0\\
829	0\\
830	0\\
831	0\\
832	0\\
833	0\\
834	0\\
835	0\\
836	0\\
837	0\\
838	0\\
839	0\\
840	0\\
841	0\\
842	0\\
843	0\\
844	0\\
845	0\\
846	0\\
847	0\\
848	0\\
849	0\\
850	0\\
851	0\\
852	0\\
853	0\\
854	0\\
855	0\\
856	0\\
857	0\\
858	0\\
859	0\\
860	0\\
861	0\\
862	0\\
863	0\\
864	0\\
865	0\\
866	0\\
867	0\\
868	0\\
869	0\\
870	0\\
871	0\\
872	0\\
873	0\\
874	0\\
875	0\\
876	0\\
877	0\\
878	0\\
879	0\\
880	0\\
881	0\\
882	0\\
883	0\\
884	0\\
885	0\\
886	0\\
887	0\\
888	0\\
889	0\\
890	0\\
891	0\\
892	0\\
893	0\\
894	0\\
895	0\\
896	0\\
897	0\\
898	0\\
899	0\\
900	0\\
901	0\\
902	0\\
903	0\\
904	0\\
905	0\\
906	0\\
907	0\\
908	0\\
909	0\\
910	0\\
911	0\\
912	0\\
913	0\\
914	0\\
915	0\\
916	0\\
917	0\\
918	0\\
919	0\\
920	0\\
921	0\\
922	0\\
923	0\\
924	0\\
925	0\\
926	0\\
927	0\\
928	0\\
929	0\\
930	0\\
931	0\\
932	0\\
933	0\\
934	0\\
935	0\\
936	0\\
937	0\\
938	0\\
939	0\\
940	0\\
941	0\\
942	0\\
943	0\\
944	0\\
945	0\\
946	0\\
947	0\\
948	0\\
949	0\\
950	0\\
951	0\\
952	0\\
953	0\\
954	0\\
955	0\\
956	0\\
957	0\\
958	0\\
959	0\\
960	0\\
961	0\\
962	0\\
963	0\\
964	0\\
965	0\\
966	0\\
967	0\\
968	0\\
969	0\\
970	0\\
971	0\\
972	0\\
973	0\\
974	0\\
975	0\\
976	0\\
977	0\\
978	0\\
979	0\\
980	0\\
981	0\\
982	0\\
983	0\\
984	0\\
985	0\\
986	0\\
987	0\\
988	0\\
989	0\\
990	0\\
991	0\\
992	0\\
993	0\\
994	0\\
995	0\\
996	0\\
997	0\\
998	0\\
999	0\\
1000	0\\
1001	0\\
1002	0\\
1003	0\\
1004	0\\
1005	0\\
1006	0\\
1007	0\\
1008	0\\
1009	0\\
1010	0\\
1011	0\\
1012	0\\
1013	0\\
1014	0\\
1015	0\\
1016	0\\
1017	0\\
1018	0\\
1019	0\\
1020	0\\
1021	0\\
1022	0\\
1023	0\\
1024	0\\
};
\end{axis}
\end{tikzpicture}%